\documentclass{article}
\usepackage[latin1]{inputenc}
\usepackage{amsmath,amsthm,amssymb}
\usepackage{amsfonts}
\usepackage{amsmath,amsthm,amssymb,amscd}
\usepackage{latexsym}
\usepackage{color}
\usepackage{graphicx}
\usepackage{mathrsfs}
\makeatletter \@addtoreset{equation}{section} \makeatother

\newcommand{\tcr}{\textcolor{red}}
\newtheorem{theorem}{Theorem}[section]

\newtheorem{proposition}{Proposition}[section]
\newtheorem{lemma}{Lemma}[section]
\newtheorem{remark}{Remark}[section]

\begin{document}
\title{On the fractional Schr$\ddot{\mbox{o}}$dinger--Kirchhoff equations
with electromagnetic fields and  critical nonlinearity}
\author{Sihua Liang$^a$, Du\v{s}an Repov\v{s}$^{b}$ and Binlin Zhang$^{c, }$\footnote{Corresponding author.
E-mail address:\, liangsihua@126.com (S. Liang), dusan.repovs@guest.arnes.si (D. Repov\v{s}), zhangbinlin2012@163.com (B. Zhang)}\\
\footnotesize $^a$ College of Mathematics, Changchun Normal
University, Changchun, 130032, P.R. China\\
\footnotesize $^b$ Faculty of Education and Faculty of Mathematics and Physics,
University of Ljubljana,\\ \footnotesize
Ljubljana, 1000, Slovenia\\
\footnotesize $^c$ Department of Mathematics, Heilongjiang Institute of Technology, Harbin,
150050, P.R. China}
\date{ }
\maketitle

\begin{abstract}
{In this paper, we consider the fractional
Schr$\ddot{\mbox{o}}$dinger--Kirchhoff  equations  with electromagnetic fields
and  critical nonlinearity
\begin{equation*}
\left\{
\begin{array}{lll}
\varepsilon^{2s}M([u]_{s,A_\varepsilon}^2)(-\Delta)_{A_\varepsilon}^su
+ V(x)u = |u|^{2_s^\ast-2}u + h(x,|u|^2)u,\ \ \
x\in  \mathbb{R}^N,\smallskip\smallskip\\
u(x) \rightarrow 0,\ \  \quad \mbox{as}\ |x| \rightarrow \infty,
\end{array}\right.
\end{equation*}
where $(-\Delta)_{A_\varepsilon}^s$ is the fractional magnetic
operator with $0<s<1$, $2_s^\ast = 2N/(N-2s)$, $M :
\mathbb{R}^{+}_{0} \rightarrow \mathbb{R}^{+}$ is a continuous
nondecreasing function, $V:\mathbb{R}^N \rightarrow \mathbb{R}^+_0$
and $A: \mathbb{R}^N \rightarrow \mathbb{R}^N$  are the electric and
magnetic potentials, respectively. By using the fractional version
of the concentration compactness principle and variational methods, we
show that the above problem: (i) has at least one solution provided that
$\varepsilon < \mathcal {E}$; and (ii) for any $m^\ast \in \mathbb{N}$,
has $m^\ast$ pairs of solutions if $\varepsilon < \mathcal
{E}_{m^\ast}$, where $\mathcal {E}$ and $\mathcal {E}_{m^\ast}$ are
sufficiently small positive numbers. Moreover, these solutions
$u_\varepsilon \rightarrow 0$ as $\varepsilon \rightarrow
0$.}\medskip

\medskip

\end{abstract}

\section{Introduction}

The main purpose of this paper is to study the existence and
multiplicity of  solutions for the fractional
Schr$\ddot{\mbox{o}}$dinger--Kirchhoff equations  with external magnetic operator and critical nonlinearity
\begin{equation}\label{e1.1}
\left\{
\begin{array}{lll}
\varepsilon^{2s}M([u]_{s,A_\varepsilon}^2)(-\Delta)_{A_\varepsilon}^su
+ V(x)u = |u|^{2_s^\ast-2}u + h(x,|u|^2)u,\
x\in  \mathbb{R}^N,\smallskip\smallskip\\
u(x) \rightarrow 0, \quad \mbox{as}\ |x| \rightarrow \infty,
\end{array}\right.
\end{equation}
where $\varepsilon > 0$ is a positive parameter, $N >2s$, $0 < s <
1$,
\begin{eqnarray*}
[u]_{s,A_\varepsilon}^2 :=
\iint_{\mathbb{R}^{2N}}\frac{|u(x)-e^{i(x-y)\cdot
A_\varepsilon(\frac{x+y}{2})}u(y)|^2}{|x-y|^{N+2s}}dx dy,
\end{eqnarray*}
where $2_s^\ast = \frac{2N}{N-2s}$ is the critical Sobolev exponent,
$V\in C(\mathbb{R}^N, \mathbb{R}^+_0)$ is the electric potential,
$A\in C(\mathbb{R}^N, \mathbb{R}^N)$  is a magnetic potential,
and $A_\varepsilon(x) := \varepsilon^{-1}A(x)$.
Further assumptions for the functions  $V(x)$, $M(x)$ and
$h(x)$  will be given in Section 3. If $A$ is a smooth
function, the fractional operator $(-\Delta)_{A}^s$, which up to
normalization constants can be defined on smooth functions $u$ as
\begin{eqnarray*}
(-\Delta)_{A}^s u(x) := 2\lim_{\varepsilon \rightarrow 0}
\int_{\mathbb{R}^N \setminus
B_\varepsilon(x)}\frac{u(x)-e^{i(x-y)\cdot
A(\frac{x+y}{2})}u(y)}{|x-y|^{N+2s}}dy, \quad x\in  \mathbb{R}^N,
\end{eqnarray*}
has recently been introduced in \cite{fel}. Hereafter,
$B_\varepsilon(x)$ denotes the ball in $\mathbb{R}^N$ centered at $x
\in \mathbb{R}^N$ and of radius $\varepsilon > 0$. As stated in
\cite{sq1}, up to correcting the operator by the factor $(1-s)$, it
follows that $(-\Delta)^s_A u$ converges to $-(\nabla u-{\rm i}
A)^2u$ as $s\rightarrow1$. Thus, up to normalization, the nonlocal
case can be seen as an approximation of the local one. The
motivations for its introduction are described in
more details in \cite{fel,sq1} and rely essentially on the L\'{e}vy-Khintchine formula
for the generator of a general L\'{e}vy process. If the magnetic
field $A \equiv 0$, then the operator $(-\Delta)_{A_\varepsilon}^s$
can be reduced to the fractional Laplacian operator $(-\Delta)^s$,
which is defined as
\begin{displaymath}
(-\Delta)^su := P.V.
\int_{\mathbb{R}^N}\frac{|u(x)-u(y)|}{|x-y|^{N+2s}} dy, \  \ \  x \in
\mathbb{R}^N,
\end{displaymath}
where $P.V.$ stands for the principal value. It may be viewed as the
infinitesimal generator of a L$\acute{\mbox{e}}$vy stable diffusion
processes \cite{ap}. This operator arises in the description of
various phenomena in applied sciences, such as phase
transitions, materials science, conservation laws, minimal surfaces,
water waves, optimization, plasma physics and so on, see \cite{di}
and references therein for more detailed introduction. Indeed, the
study of fractional and nonlocal operators of elliptic type  has recently
attracted more attention. For example, for the case in
which bounded domains and the entire space are involved,
we refer the readers to \cite{ADM, BMS, m0,m2, YW, ZDM} and the
references therein for more related results.

The main driving force for the study of problem (\ref{e1.1}) arises in the following time-dependent Schr\"{o}dinger equation when $s=1$:
\begin{equation}\label{101}
i\hbar \frac{\partial\psi}{\partial t}=\frac{1}{2m}(-i\hbar \nabla +A(x))^2\psi+P(x)\psi-\rho(x, |\psi|)\psi,
\end{equation}
where $\hbar$ is the Planck constant, $m$ is the particle mass, $A:\mathbb{R}^{N}\rightarrow\mathbb{R}^{N}$ is the magnetic potential, $P:\mathbb{R}^{N}\rightarrow\mathbb{R}^{N}$ is the electric potential,
$\rho$ is the nonlinear coupling, and $\psi$ is the wave function representing the state of the particle.
This equation arises in quantum
mechanics and describes the dynamics of the particle in a non-relativistic
setting, see for example \cite{AAHM, RS}. Clearly, the form $\psi(x,t)=e^{-i\omega t \hbar^{-1}}u(x)$
is a standing wave solution of \eqref{101} if and only if $u(x)$ satisfies the following stationary equation:
\begin{align}\label{102}
(-i\varepsilon\nabla+A)^{2}u+V(x)u=f(x, |u|)u,
\end{align}
where $\varepsilon=\hbar$, $V(x)=2m(P(x)-\omega)$ and $f=2m\rho$, see \cite{BNV, DV, DW, FTRV} and the references cited therein for recent results in this direction.
When $A \equiv 0$, problem
\eqref{102} becomes the classical Schr$\ddot{\mbox{o}}$dinger
equation
\begin{equation}
- \varepsilon^2\Delta u + V(x)u = f(x,u),\ \  x \in \mathbb{R}^N.
\end{equation}
Similarly, we can
deduce
the following fractional Schr$\ddot{\mbox{o}}$dinger equation:
\begin{equation}\label{e1.4}
\varepsilon^{2s}(-\Delta)^su + V(x)u = f(x,u), \ \ \ \ x\in  \mathbb{R}^N.
\end{equation}
{\it Felmer}, {\it Quaas} and {\it Tan} \cite{fel} studied the existence and
regularity of positive solutions for problem \eqref{e1.4} with $\varepsilon=1$ when $f$ has subcritical growth and
satisfies the Ambrosetti-Rabinowitz condition. {\it Secchi} in \cite{sec}
obtained the existence of ground state solutions of \eqref{e1.4}
when $V(x)\rightarrow \infty$ as $|x| \rightarrow \infty$ and
the Ambrosetti-Rabinowitz condition holds.  {\it Dong}, {\it Xu} and {\it Wei} \cite{we1}
obtained the existence of infinitely many weak solutions for
\eqref{e1.4} by a variant of the fountain theorem when $f$ has subcritical
growth. For the case of critical growth, {\it Shang} and {\it Zhang} \cite{sh1}
studied the existence and multiplicity of solutions for the critical
fractional Schr$\ddot{\mbox{o}}$dinger equation:
\begin{equation}\label{e1.5}
\varepsilon^{2s}(-\Delta)^s u + V(x)u = |u|^{2_s^\ast-2}u + \lambda
f(u) \ \ \ x\in  \mathbb{R}^N.
\end{equation}
Based on variational methods, they showed that problem \eqref{e1.5}
has a nonnegative ground state solution for all sufficiently large
$\lambda$ and small $\varepsilon$. Moreover, {\it Shen} and {\it Gao}
\cite{sh3} proved the existence of nontrivial solutions for
problem \eqref{e1.5} under various assumptions on $f$ and potential
function $V(x)$, among which they also assumed the well-known
Ambrosetti-Rabinowitz condition. See also recent papers
\cite{AH, b2, sec,sh2} for more results. {\it Teng} and {\it He} \cite{te1}, were concerned
with the following fractional Schr$\ddot{\mbox{o}}$dinger equation
involving a critical nonlinearity
\begin{equation}\label{e1.6}
\varepsilon^{2s}(-\Delta)^s u + u = Q(x)|u|^{2_s^\ast-2}u
+P(x)|u|^{p-2}u,\ \ \ \ x\in  \mathbb{R}^N,
\end{equation}
where $2 < p < 2_s^\ast$ and potential functions $P(x)$ and $Q(x)$
satisfy certain hypotheses. Using the s-harmonic extension technique
of {\it Caffarelli} and {\it Silvestre} \cite{ca}, the concentration-compactness
principle of {\it Lions} \cite{lions1} and methods of
{\it Br$\acute{\mbox{e}}$zis} and {\it Nirenberg} \cite{bre}, they
proved the existence of ground state solutions.
On the other hand, {\it Feng} \cite{feng} investigated the
following fractional Schr$\ddot{\mbox{o}}$dinger equation
\begin{equation}\label{e1.7}
(-\Delta)^s u + V(x)u = \lambda|u|^{p}u\ \ \ \ x\in \mathbb{R}^N,
\end{equation}
where $2 < p < 2_s^\ast$ and $V(x)$ is a positive continuous function.
By using the fractional version of concentration compactness
principle of {\it Lions}  \cite{lions1}, he obtained the existence
of ground state solutions to problem \eqref{e1.7} for some $\lambda
> 0$. By applying another fractional version of concentration compactness
principle and radially decreasing rearrangements, {\it Zhang et al.}  \cite{zhang1} proved the existence
of a ground state solutions for problem \eqref{e1.5} with $V(x)=1$ for large enough $\lambda>0$, see \cite{zhang2} for related result with  application of the same method.

Another important reason for studying problem \eqref{e1.1}
lies in the following feature of the Kirchhoff problems. More precisely, {\it Kirchhoff} proposed the following model in 1883
\begin{align}\label{eq3}
\rho\frac{\partial ^2u}{\partial t^2}-\left(\frac{p_0}{\lambda}+\frac{E}{2L}\int_0^L\left|\frac{\partial u}{\partial x}\right|^2dx\right)\frac{\partial ^2u}{\partial x^2}=0
\end{align}
as a generalization of the well-known D'Alembert's wave equation for free vibrations of elastic strings. Here, $L$ is the length of the string, $h$ is the area of the cross section, $E$ is the Young modulus of the material, $\rho$ is the mass density and $p_0$ is the initial tension.
Essentially, Kirchhoff's model takes into account the changes in the length of the string produced by transverse vibrations.
For recent results in this direction, for example, we refer the reader to \cite{liang2, liang1} and references therein.
Recently, {\it Fiscella} and {\it Valdinoci} \cite{fi} first deduced a stationary fractional Kirchhoff
model which considered the nonlocal
aspect of the tension arising from nonlocal measurements of the
fractional length of the string, see the Appendix of \cite{fi} for more details.
Moreover, they investigated  in \cite{fi} also the following Kirchhoff type problem involving critical exponent:
 \begin{eqnarray}\label{eq13}
\begin{cases}
M([u]^2_s)(-\Delta)^su=\lambda f(x,u)+|u|^{2_s^*-2}u\quad &\mbox{in }\,\,\Omega\\
u=0\quad&\mbox{in}\,\, \mathbb{R}^N\setminus\Omega.
\end{cases}
\end{eqnarray}
where $\Omega$ is an open
bounded domain in $\mathbb{R}^N$. By using the mountain pass theorem and the concentration compactness principle, together with a truncation technique, they obtained the existence of non-negative solutions for problem \eqref{eq13},
see for example \cite{au, pu, XZR} for more recent results.
For the results on the entire space, see for instance \cite{f3, PXZ, PXZ1}.

{\it Mingqi} {\it et al.} \cite{MPSZ} first studied the following Schr$\ddot{\mbox{o}}$dinger--Kirchhoff type equation
involving the fractional $p$--Laplacian and the magnetic operator
\begin{equation}\label{eq11}
M([u]_{s,A}^2)(-\Delta)_A^su+V(x)u=f(x,|u|)u\quad \text{in $\mathbb{R}^N$},
\end{equation}
where the right-hand term in \eqref{eq11} satisfies the subcritical growth.
By using  variational methods, they obtained several existence results for problem \eqref{eq11}.
Following  similar methods, for $M(t)=a+bt$ with $a\in \mathbb{R}^+_0$ and $p=2$,
{\it Wang} and {\it Xiang} \cite{WX} proved the existence of two solutions and  infinitely many solutions for fractional Schr$\ddot{\mbox{o}}$dinger-Choquard-Kirchhoff type equations
with external magnetic operator and critical exponent in the sense of the Hardy-Littlewood-Sobolev inequality.
{\it Binlin} {\it et al.} \cite{zhang3} first considered the following fractional Schr$\ddot{\mbox{o}}$dinger equations:
\begin{eqnarray}\label{eq12}
\varepsilon^{2s}(-\Delta)^{s}_{A_{\varepsilon}}u+V(x)u=f(x,|u|)u+K(x)|u|^{2_{\alpha}^{*}-2}u\quad\quad
\mbox{in}\ \mathbb{R}^{N},
\end{eqnarray}
where $V(x)$ satisfies the assumption $(V)$ which will be introduced in Section 3.
By using variational methods, they proved the existence of  ground state solution (mountain pass solution) $u_{\varepsilon}$ which tends to the trivial solution as $\varepsilon\rightarrow0$. Moreover, they proved the existence of infinite many solutions and sign-changing solutions for problem \eqref{eq12} under some additional assumptions.

Inspired by the above works, in particular by \cite{zhang3, d3, MPSZ, zhang2}, we
consider in this article the existence and multiplicity of semiclassical solutions
of the fractional Schr$\ddot{\mbox{o}}$dinger--Kirchhoff equations with
electromagnetic fields and  critical nonlinearity in $\mathbb{R}^N$.
It is worthwhile to remark that in the arguments developed in
\cite{zhang3, d3}, one of the key points is to prove the  $(PS)_c$
condition. Here we use the fractional version of Lions' second
concentration compactness principle  to prove that the $(PS)_c$ condition holds,
which is different from methods used in \cite{zhang3, d3}. Some
difficulties arise when dealing with this problem, because of the
appearance of the magnetic field and the critical frequency, and of
the nonlocal nature of the fractional Laplacian. Therefore, we need
to develop new techniques to overcome difficulties induced by these
new features. As far as we know, this is the first time that the
fractional version of the concentration compactness principle and
variational methods have been combined to get  the multiplicity of solutions for the fractional Schr$\ddot{\mbox{o}}$dinger--Kirchhoff  equations  with
electromagnetic fields and  critical nonlinearity. We
believe that the ideas used here can be applied in other situations
to deal with  similar potentials.

The paper is organized as follows. In
Section 2, we will introduce the working space and give some
necessary definitions and properties, which will be used in the
sequel. In Section 3,  we will give an equivalent form of problem
\eqref{e1.1}. In Section 4, we will use the fractional version of
Lions' second 
concentration compactness principle to prove that the $(PS)_c$
condition holds. In Section 5, using the critical point theory, we will prove the main result (see Section 3).

\section{Preliminaries}\label{sec2}

For the convenience of the reader, we recall in this part
some definitions and basic properties of fractional Sobolev spaces
$H_{A_\varepsilon}^{s}(\mathbb{R}^N,\mathbb{C})$. For a deeper
treatment of the (magnetic) fractional Sobolev spaces and their applications to fractional
Laplacian problems of elliptic
type, we refer to \cite{zhang3, di, f2, MPSZ, m4, PSV1, PSV2} and the references therein. \\
\indent For any $s \in (0, 1)$, the fractional Sobolev space
$H_{A_\varepsilon}^{s}(\mathbb{R}^N,\mathbb{C})$ is defined by
\begin{displaymath}
H_{A_\varepsilon}^{s}(\mathbb{R}^N,\mathbb{C})  =  \left\{u \in
L^2(\mathbb{R}^N,\mathbb{C}): [u]_{s,A_\varepsilon} <
\infty\right\},
\end{displaymath}
where $[u]_{s,A_\varepsilon}$ denotes the so-called Gagliardo
semi-norm, that is
\begin{displaymath}
[u]_{s,A_\varepsilon}  =
\left(\iint_{\mathbb{R}^{2N}}\frac{|u(x)-e^{i(x-y)\cdot
A_\varepsilon(\frac{x+y}{2})}u(y)|^2}{|x-y|^{N+2s}}dx
dy\right)^{1/2}
\end{displaymath}
and $H_{A_\varepsilon}^{s}(\mathbb{R}^N,\mathbb{C})$ is endowed with
the norm
\begin{displaymath}
\|u\|_{H_{A_\varepsilon}^{s}(\mathbb{R}^N,\mathbb{C})} =
\left([u]_{s,A_\varepsilon}^2 + \|u\|_{L^2}^2\right)^{\frac{1}{2}}.
\end{displaymath}
If $A=0$, then $H_{A_\varepsilon}^{s}(\mathbb{R}^N,\mathbb{C})$
reduces to the well-known space $H^{s}(\mathbb{R}^N)$ with the norm $[u]_s:=[u]_{s,0}$. The
space $H_{A_\varepsilon}^{s}(\mathbb{R}^N,\mathbb{C})$ is also a Hilbert
space with the real scalar product
\begin{displaymath}
\langle u, v\rangle_{s,A_\varepsilon} := \langle u, v\rangle_{L^2} +
\mbox{Re}\iint_{\mathbb{R}^{2N}}\frac{(u(x)-e^{i(x-y)\cdot
A_\varepsilon(\frac{x+y}{2})}u(y))\overline{(v(x)-e^{i(x-y)\cdot
A_\varepsilon(\frac{x+y}{2})}v(y))}}{|x-y|^{N+2s}}dx dy,
\end{displaymath}
for any $u,v \in H_{A_\varepsilon}^{s}(\mathbb{R}^N,\mathbb{C})$.
The operator $((-\Delta)_{A_\varepsilon}^s):
H_{A_\varepsilon}^{s}(\mathbb{R}^N,\mathbb{C}) \rightarrow
H_{A_\varepsilon}^{-s}(\mathbb{R}^N,\mathbb{C})$ is defined by
\begin{displaymath}
\langle (-\Delta)_{A_\varepsilon}^su, v\rangle :=
\mbox{Re}\iint_{\mathbb{R}^{2N}}\frac{(u(x)-e^{i(x-y)\cdot
A_\varepsilon(\frac{x+y}{2})}u(y))\overline{(v(x)-e^{i(x-y)\cdot
A_\varepsilon(\frac{x+y}{2})}v(y))}}{|x-y|^{N+2s}}dx dy,
\end{displaymath}
via duality. \\
\indent We recall the following embedding theorem:
\begin{proposition}\label{pro1}{\em (See \cite[Lemma 3.5]{fel})}. Let $A\in C(\mathbb{R}^N, \mathbb{R}^N)$. Then the embedding
\begin{displaymath}
H_{A_\varepsilon}^{s}(\mathbb{R}^N,\mathbb{C}) \hookrightarrow
L^{\theta}(\mathbb{R}^N,\mathbb{C}),
\end{displaymath}
 is continuous for any $\theta \in [2,2_s^\ast]$. Moreover, the
 embedding
\begin{displaymath}
H_{A_\varepsilon}^{s}(\mathbb{R}^N,\mathbb{C})
\hookrightarrow\hookrightarrow L^{\theta}_{{\rm loc}}(\mathbb{R}^N,\mathbb{C})
\end{displaymath}
is compact for any $\theta \in [1,2_s^\ast)$.
\end{proposition}
\indent In this paper, we will use the following subspace of
$H_{A_\varepsilon}^{s}(\mathbb{R}^N,\mathbb{C})$ defined by
\begin{displaymath}
E =\left\{u\in H_{A_\varepsilon}^{s}(\mathbb{R}^N,\mathbb{C}):
\int_{\mathbb{R}^N}V(x)|u|^{2}dx<\infty \right\}
\end{displaymath}
with the norm
\begin{displaymath}
\|u\|_E := \left([u]_{s,A_\varepsilon}^2 +
\int_{\mathbb{R}^N}V(x)|u|^2dx\right)^{\frac{1}{2}},
\end{displaymath}
where $V$ is non-negative. By the assumption $(V)$ (see Section 3), we know that the
embedding $E \hookrightarrow
H_{A_\varepsilon}^{s}(\mathbb{R}^N,\mathbb{C})$ is continuous. Note
that the norm $\|\cdot\|_E$ is equivalent to the norm
$\|\cdot\|_\varepsilon$ defined by
\begin{displaymath}
\|u\|_\varepsilon := \left([u]_{s,A_\varepsilon}^2 +
\varepsilon^{-2s}\int_{\mathbb{R}^N}V(x)|u|^2dx\right)^{\frac{1}{2}},
\end{displaymath}
for each $\varepsilon > 0$.  It is obvious that for each $ \theta
\in [2, 2_s^\ast]$, there is $c_{\theta}>0$, independent of $0 <
\varepsilon < 1$, such that
\begin{equation}\label{e2.2}
\|u\|_{L^\theta} \leq c_{\theta}\|u\|_E \leq
c_{\theta}\|u\|_{\varepsilon}.
\end{equation}
\indent We have the following diamagnetic inequality:
\begin{lemma}  For every $u \in
H_{A_\varepsilon}^{s}(\mathbb{R}^N,\mathbb{C})$, we get $|u| \in
H^{s}(\mathbb{R}^N)$. More precisely,
\begin{displaymath}
[|u|]_{s} \leq [u]_{s,A_\varepsilon}.
\end{displaymath}
\end{lemma}
\begin{proof}
The assertion follows directly from the pointwise diamagnetic
inequality
\begin{displaymath}
||u(x)|-|u(y)|| \leq |u(x) - e^{i(x-y)\cdot
A_\varepsilon(\frac{x+y}{2})}u(y)|,
\end{displaymath}
for a.e. $x,y \in \mathbb{R}^N$, see \cite[Lemma 3.1, Remark
3.2]{fel}.
\end{proof}
\indent By Proposition 3.6 in \cite{di}, we have
\begin{displaymath}
[u]_{s} = \|(-\Delta)^{\frac{s}{2}}u\|_{L^2}
\end{displaymath}
for any $u \in H^{s}(\mathbb{R}^N)$, i.e.
\begin{displaymath}
\iint_{\mathbb{R}^{2N}}\frac{|u(x)-u(y)|^2}{|x-y|^{N+2s}}dxdy  =
\int_{\mathbb{R}^{N}}|(-\Delta)^{\frac{s}{2}}u(x)|^2dx.
\end{displaymath}
Thus
\begin{displaymath}
\iint_{\mathbb{R}^{2N}}\frac{(u(x)-u(y))(v(x)-v(y))}{|x-y|^{N+2s}}dx
dy  =
\int_{\mathbb{R}^{N}}(-\Delta)^{\frac{s}{2}}u(x)\cdot(-\Delta)^{\frac{s}{2}}v(x)
dx.
\end{displaymath}

\section{The main result}

\indent Throughout the paper,   without explicit mention, we suppose
that the functions $V(x)$, $M(x)$ and $h(x)$  satisfy the
following conditions:
\begin{itemize}
\item[($V$)]  $V(x) \in C(\mathbb{R}^N, \mathbb{R})$, $\min_{x\in \mathbb{R}^N} V(x) = 0$ and there is $\tau_0 > 0$ such that the set $V^{\tau_0} = \{x \in \mathbb{R}^N: V(x) < \tau_0\}$
has finite Lebesgue measure;
\item[($M$)] ($m_1$) $M : \mathbb{R}^{+}_{0} \rightarrow \mathbb{R}^{+}$ is a continuous nondecreasing function. Furthermore, there exists $\alpha_0>0$ such that
and  $M(t)\geq\alpha_0$ for all $t\in \mathbb{R}^+_0$;\\
($m_2$) there exists $\sigma\in (2/2_s^\ast, 1]$ satisfying $\widetilde{M}(t)\geq\sigma M(t)t$ for all $t\geq0$, where
$\widetilde{M}(t)=\int_0^tM(s)ds$;
\item[($H$)]  ($h_1$) $h \in C(\mathbb{R}^N \times \mathbb{R}, \mathbb{R})$ and $h(x, t) =
o(|t|)$ uniformly in $x$ as $t \rightarrow 0$; \\
($h_2$) there exist $c_0 > 0$ and $q \in (2,2_s^\ast)$ such that
$|h(x, t)|
\leq c_0 (1 + t^{\frac{q-1}{2}})$;\\
($h_3$)  there exist $l_0 > 0$, $2/\sigma <r$ and $2/\sigma<  \mu <
2_s^\ast$ such that $H(x, t) \geq l_0 |t|^{r/2}$ and
$\mu H(x, t) \leq 2h(x, t)t$ for all $(x, t)\in \mathbb{R}^N \times \mathbb{R}$, where $H(x,
t) = \int_0^t h(x, s)ds$.
\end{itemize}
To obtain the solution of problem \eqref{e1.1}, we will use the
following equivalent form
\begin{equation}\label{e3.1}
\left\{
\begin{array}{lll}M\Big([u]_{s,A_\varepsilon}^2\Big)(-\Delta)_{A_\varepsilon}^su + \varepsilon^{-2s} V(x)u = \varepsilon^{-2s}|u|^{2_s^\ast-2}u + \varepsilon^{-2s}h(x,|u|^2)u, \,
x\in  \mathbb{R}^N,\smallskip\smallskip\\
u(x) \rightarrow 0, \ \indent as \ \ |x| \rightarrow \infty,
\end{array}\right.
\end{equation}
for $\varepsilon \rightarrow 0$. \\
\indent The energy functional $J_\varepsilon: E \rightarrow
\mathbb{R}$ associated with problem  \eqref{e3.1}
\begin{eqnarray*}
J_\varepsilon(u) :=\frac12 \widetilde{M}\left([u]_{s,A_\varepsilon}^2\right)+
\frac{\varepsilon^{-2s}}{2}\int_{\mathbb{R}^N}V(x)|u|^2dx -
\frac{\varepsilon^{-2s}}{2_s^\ast}\int_{\mathbb{R}^N}
|u|^{2_s^\ast}dx
  - \frac{\varepsilon^{-2s}}{2}\int_{\mathbb{R}^N}H(x,|u|^2)dx
\end{eqnarray*}
is well defined. Define the Nehari manifold
\begin{eqnarray*}
\mathcal {N} = \left\{u \in E: \langle J_\varepsilon'(u), u\rangle_E
= 0 \right\}.
\end{eqnarray*}
Under the assumptions, it is easy to check that as shown
in \cite{r1,w1}, $J_\varepsilon \in C^1 (E, \mathbb{R})$ and its
critical points are weak solutions of problem  \eqref{e3.1}.
\\ \indent We recall that $u \in
E$ is a weak solution of problem  \eqref{e3.1}, if
\begin{eqnarray*}
&&
M\left([u]_{s,A_\varepsilon}^2\right)\mbox{Re}\iint_{\mathbb{R}^{2N}}\frac{(u(x)-e^{i(x-y)\cdot
A_\varepsilon(\frac{x+y}{2})}u(y))\overline{(v(x)-e^{i(x-y)\cdot
A_\varepsilon(\frac{x+y}{2})}v(y))}}{|x-y|^{N+2s}}dx dy  \\
&& \mbox{}\ + \varepsilon^{-2s}\mbox{Re}
\int_{\mathbb{R}^N}V(x)u\bar{v}dx = \varepsilon^{-2s} \mbox{Re}
\int_{\mathbb{R}^N}\left(|u|^{2_s^\ast-2}u + h(x,
|u|^2)u\right)\bar{v} dx,
\end{eqnarray*}
where $v \in E$.\\
 \indent The following is the main result of the present paper. It will be proved in Section 5.
\begin{theorem}\label{the3.1} Let the conditions {\rm ($V$), ($M$)} and {\rm ($H$)}  be satisfied. Then the following statements hold:\\
$(1)$ For any $\kappa > 0$, there is $\mathcal {E}_\kappa > 0$ such
that if $0 < \varepsilon <  \mathcal {E}_\kappa$, then problem
\eqref{e3.1} has at least one solution $u_\varepsilon$  satisfying
\begin{eqnarray}\label{e1.8}
\frac{\sigma\mu-1}{2}\int_{\mathbb{R}^N}H(x,|u_{\varepsilon}|^2)dx
+\left(\frac{2}{\sigma}-\frac{1}{2_s^\ast}\right)\int_{\mathbb{R}^N}
|u_{\varepsilon}|^{2_s^\ast}dx\leq \kappa{\varepsilon^{N}},
\end{eqnarray}
\begin{eqnarray}\label{e1.9}
\left(\frac{\sigma}{2}-\frac{1}{\mu}\right)\alpha_0\varepsilon^{2s}
[u_{\varepsilon}]_{s,A}^2 + \left(\frac{1}{2}-\frac{1}{\mu}\right)
\int_{\mathbb{R}^N}V(x)|u_{\varepsilon}|^2dx \leq
\kappa{\varepsilon^{N}}.
\end{eqnarray}
Moreover, $u_\varepsilon \rightarrow 0$ in $E$ as $\varepsilon
\rightarrow
0$.

\noindent $(2)$ For any $m \in \mathbb{N}$ and $\kappa > 0$, there is
$\mathcal {E}_{m\kappa}
> 0$ such that if $0 < \varepsilon < \mathcal {E}_{m\kappa}$, then problem \eqref{e3.1} has at least $m$ pairs of
solutions $u_{\varepsilon,i}$, $u_{\varepsilon,-i}$,
$i=1,2,\cdots,m$ which satisfy the estimates \eqref{e1.8} and
\eqref{e1.9}. Moreover, $u_{\varepsilon,i}\rightarrow 0$ in $E$ as
$\varepsilon \rightarrow 0$, $i=1,2,\cdots,m$.
\end{theorem}

%%%%%%%%%%%%%%%%%%%%%%%%%%%%%%%%%%%%%%%%%%%%%%%%%%%%%%%%%%%%%%%%%%%%%%%%%%%%%%
%%%%%%%%%%%%%%%%%%%%%%%%%%%%%%%%%%%%%%%%%%%%%%%%%%%%%%%%%%%%%%%%%%%%%%%%%%%%%%%%%%%
\section{Behaviour of (PS) sequences}

In this section, we recall the fractional version of concentration compactness
principle in the fractional Sobolev space, see \cite{PP, XZZ, zhang1} for more details.
Note that Prokhorov theorem (see Theorem
8.6.2 in \cite{Bog}) ensures that bounded sequences $\{u_n\}_n$ are relatively sequentially compact in
$H^{s}(\mathbb{R}^N)$ if and only if
the sequence is {\it tight} in the sense that
$\mbox{for any}\  \varepsilon >0, \  \mbox{there exists} \  \mbox{a compact subset}\ \Omega\subseteq \mathbb{R}^N\
\mbox{such that}\  \sup_{n}\int_{\mathbb{R}^N\setminus \Omega}|u_n| dx<\varepsilon.$

\begin{lemma}{\em (\cite[Theorem 1.5]{PP})}\label{lemma3.1} Let $\Omega \subseteq \mathbb{R}^N$ be an open subset and let $\{u_n\}_n$ be a weakly convergent sequence  in
$H^s(\mathbb{R}^N)$, weakly converging to $u$ as $n \rightarrow
\infty$ and such that $|u_n|^{2_s^\ast}\rightharpoonup \nu$ and
$|(-\Delta)^{\frac{s}{2}} u_n|^2 \rightharpoonup \eta$ in the sense
of measures. Then either $u_n \rightarrow u$ in
$L_{{\rm loc}}^{2_s^\ast}(\mathbb{R}^N)$ or there exist a (at most
countable) set of distinct points $\{x_j\}_{j \in I} \subseteq
\overline{\Omega}$ and positive numbers $\{\nu_j\}_{j \in I}$ such
that
\begin{eqnarray*}
\nu = |u|^{2_s^\ast} + \sum_{j \in I} \delta_{x_j}\nu_j,\quad \nu_j
> 0.
\end{eqnarray*}
If, in addition, $\Omega$ is bounded, then there exist a positive
measure $\widetilde{\eta} \in \mathcal {M}(\mathbb{R}^N)$ with
$supp~\widetilde{\eta} \subseteq \overline{\Omega}$ and positive
numbers $\{\eta_j\}_{j \in I}$ such that
\begin{eqnarray*}
\eta =  |(-\Delta)^{\frac{s}{2}} u|^2 + \widetilde{\eta} + \sum_{j \in
I} \delta_{x_j}\eta_j, \quad  \eta_j > 0
\end{eqnarray*}
and
\begin{eqnarray*}
\nu_j \leq (S^{-1}\eta(\{x_j\}))^{\frac{2_s^\ast}{2}},
\end{eqnarray*}
where $S$ is the best Sobolev constant, i.e.
\begin{eqnarray*} S = \inf\limits_{u \in
H^s(\mathbb{R}^N)}
\frac{\displaystyle\int_{\mathbb{R}^N}|(-\Delta)^{\frac{s}{2}}
u|^2dx}{\displaystyle\int_{\mathbb{R}^N}|u|^{2_s^\ast}dx},
\end{eqnarray*} $x_j \in
\mathbb{R}^N$, $\delta_{x_j}$ are Dirac measures at $x_j$, and
$\mu_j$, $\nu_j$ are constants.
\end{lemma}

\begin{remark}
 {\em In the case $\Omega = \mathbb{R}^N$, the above principle of
 concentration compactness does not provide any information about
 the possible loss of mass at infinity. The following result
 expresses this fact in quantitative terms.}
\end{remark}
\begin{lemma}{\em (\cite[Lemma 3.5]{zhang1})}\label{lemma3.2} Let $\{u_n\}_n \subset H^s(\mathbb{R}^N)$ be such that $u_n \rightharpoonup
u$ weakly converges in $H^s(\mathbb{R}^N)$,  $|u_n|^{2_s^\ast}\rightharpoonup
\nu$ and $|(-\Delta)^{\frac{s}{2}} u_n|^2 \rightharpoonup \eta$
weakly-$\ast$ converges in $\mathcal {M}(\mathbb{R}^N)$ and define
\begin{itemize}
\item[$\mathrm{(i)}$]  $\eta_\infty = \lim\limits_{R\rightarrow
\infty}\limsup\limits_{n\rightarrow\infty}\displaystyle\int_{\{x \in
\mathbb{R}^N: |x|>R\}}|(-\Delta)^{\frac{s}{2}} u_n|^{2}dx$,

\item[$\mathrm{(ii)}$]   $\nu_\infty = \lim\limits_{R\rightarrow
\infty}\limsup\limits_{n\rightarrow\infty}\displaystyle\int_{\{x \in
\mathbb{R}^N: |x|>R\}}|u_n|^{2_s^\ast}dx$.
\end{itemize}
Then the quantities $\nu_\infty$ and $\eta_\infty$ exist and satisfy the following
\begin{itemize}
\item[$\mathrm{(iii)}$]  $\limsup\limits_{n\rightarrow\infty}\displaystyle\int_{\mathbb{R}^N}|(-\Delta)^{\frac{s}{2}}
u_n|^{2}dx = \displaystyle\int_{\mathbb{R}^N}d\eta + \eta_\infty$,

\item[$\mathrm{(iv)}$]   $\limsup\limits\limits_{n\rightarrow\infty}\displaystyle\int_{\mathbb{R}^N}|u_n|^{2_s^\ast}dx
= \displaystyle\int_{\mathbb{R}^N}d\nu + \nu_\infty$,

\item[$\mathrm{(v)}$] $\nu_\infty \leq (S^{-1}\eta_\infty)^{\frac{2_s^\ast}{2}}$.
\end{itemize}
\end{lemma}
We recall that a $C^1$ functional $J$ on Banach space $X$ is said to
satisfy the Palais-Smale condition at level $c$ $((PS)_c$ in short)
if every sequence $\{u_n\}_n \subset X$ satisfying
$\lim\limits_{n\rightarrow\infty}J_\lambda(u_n) = c$ and
$\lim\limits_{n\rightarrow\infty}\|J_\lambda'(u_n)\|_{X^\ast} = 0$
has a convergent subsequence.
\begin{lemma}\label{lemma3.3}
Suppose that conditions {\rm ($V$)}, {\rm ($M$)} and {\rm ($H$)}  hold. Then
any $(PS)_c$ sequence $\{u_n\}_n$ is bounded in $E$ and $c \geq 0$.
\end{lemma}
\begin{proof}
Let $\{u_n\}_n$ be a $(PS)$ sequence in $E$.  Then
\begin{eqnarray}\label{e4.1}
c &=&\nonumber  J_\varepsilon(u_n) \\
&=&
\frac12 \widetilde{M}\left([u_n]_{s,A_\varepsilon}^2\right)+
\frac{\varepsilon^{-2s}}{2}\int_{\mathbb{R}^N}V(x)|u_n|^2dx  - \frac{\varepsilon^{-2s}}{2_s^\ast}\int_{\mathbb{R}^N} |u_n|^{2_s^\ast}dx
- \frac{\varepsilon^{-2s}}{2}\int_{\mathbb{R}^N}H(x, |u_n|^2)dx,
\end{eqnarray}
\begin{eqnarray}\label{e4.2}
\langle J_\varepsilon'(u_n), v\rangle &=&\nonumber \mbox{Re}\Bigg\{
M\left([u_n]_{s,A_\varepsilon}^2\right)\iint_{\mathbb{R}^{2N}}\frac{(u_n(x)-e^{i(x-y)\cdot
A_\varepsilon(\frac{x+y}{2})}u_n(y))\overline{(v(x)-e^{i(x-y)\cdot
A_\varepsilon(\frac{x+y}{2})}v(y))}}{|x-y|^{N+2s}}dx dy\Bigg.
\\
 &&\nonumber \mbox{}\left.+ \varepsilon^{-2s}
\int_{\mathbb{R}^N}V(x)u_n\bar{v}dx - \varepsilon^{-2s}
\int_{\mathbb{R}^N}|u_n|^{2_s^\ast-2}u_n\bar{v} dx -
\varepsilon^{-2s}\int_{\mathbb{R}^N}h(x, |u_n|^2)u_n\bar{v}dx\right\} \\
 &=&  o(1)\|u_n\|_\varepsilon.
\end{eqnarray}
By \eqref{e4.1}, \eqref{e4.2}, $(M)$ and condition $(h_3)$, we have
\begin{eqnarray}\label{e4.3}
c+ o(1)\|u_n\|_\varepsilon&=&\nonumber J_\varepsilon(u_n) -
\frac{1}{\mu}\langle J_\varepsilon'(u_n), u_n\rangle =
\frac12 \widetilde{M}\left([u_n]_{s,A_\varepsilon}^2\right)-
\frac{1}{\mu}M\left([u_n]_{s,A_\varepsilon}^2\right)[u_n]_{s,A_\varepsilon}^2 \\
&&\nonumber
\mbox{}+\left(\frac{1}{2}-\frac{1}{\mu}\right)\varepsilon^{-2s}
\int_{\mathbb{R}^N}V(x)|u_n|^2dx +
\left(\frac{1}{\mu}-\frac{1}{2_s^\ast}\right)\varepsilon^{-2s}\int_{\mathbb{R}^N} |u_n|^{2_s^\ast}dx  \\
 &&\nonumber \mbox{} + \varepsilon^{-2s}\int_{\mathbb{R}^N} \left[\frac{1}{\mu}h(x, |u_n|^2)u_n^2-\frac{1}{2}H(x,|u_n|^2)\right]dx\\
&\geq& \left(\frac{\sigma}{2}-\frac{1}{\mu}\right)\alpha_0
[u_n]_{s,A_\varepsilon}^2 +
\left(\frac{1}{2}-\frac{1}{\mu}\right)\varepsilon^{-2s}
\int_{\mathbb{R}^N}V(x)|u_n|^2dx.
\end{eqnarray}
Therefore,  \eqref{e4.3}  implies that $\{u_n\}_n$ is
bounded in $E$. Passing to the limit in \eqref{e4.3} shows that $c\geq
0$.  This completes the proof.\end{proof}
\indent The main result in this section is the following compactness
result:
\begin{theorem}\label{lemma3.4}Suppose that conditions {\rm ($V$),  ($M$)} and {\rm ($H$)}  hold. Then for any $0 < \varepsilon < 1$, $J_\varepsilon$
satisfies $(PS)_c$ condition, for all $c \in \left(0,\,
 \sigma_0\varepsilon^{N-2s}\right)$, where $\sigma_0 :=
\left(\frac{1}{\mu}-\frac{1}{2_s^\ast}\right)(\alpha_0S)^{N/(2s)}$,
that is any $(PS)_c$-sequence $\{u_n\}_n \subset E$ has a strongly
convergent subsequence in $E$.
\end{theorem}
\begin{proof}
 Let $\{u_n\}_n$ be a $(PS)_c$ sequence. By Lemma \ref{lemma3.3},  $\{u_n\}_n$ is
bounded in $E$. Hence, by diamagnetic inequality, $\{|u_n|\}_n$ is
bounded in $H^s(\mathbb{R}^N)$. Then, for some subsequence, there is
$u \in E$ such that $u_n \rightharpoonup u$ in $E$. We claim that as $n\rightarrow \infty$
\begin{equation} \label{e4.4}
\int_{\mathbb{R}^N} |u_n|^{2_s^\ast} dx \rightarrow
\int_{\mathbb{R}^N} |u|^{2_s^\ast} dx.
\end{equation}
In order to prove this claim, we invoke Prokhorov's Theorem (see Theorem
8.6.2 in \cite{Bog}) to conclude that there exist $\tcr{\eta}, \nu \in \mathcal {M}(\mathbb{R}^N)$
such that
\begin{eqnarray*}
\begin{gathered}
|(-\Delta)^{\frac{s}{2}}
u_n|^2 \rightharpoonup \eta \quad (\text{weak*-sense of measures}),\\
|u_n|^{2_s^\ast}\rightharpoonup \nu\quad (\text{weak*-sense of
measures}),
\end{gathered}
\end{eqnarray*}
where $\mu$ and $\nu$ are a nonnegative  bounded measures on
$\mathbb{R}^N$. For this, we have to show the {\it tightness} of sequences
$\{|(-\Delta)^{s/2}u_n|^2\}_n$ and $\{|u_n|^{2_s^\ast}\}_n$, which follows easily from the boundedness of $\{|u_n|\}_n$ in $H^s(\mathbb{R}^N)$ and absolute continuity of the Lebesgue integral.
Then, in view of Lemma \ref{lemma3.1}, we know that
either $u_n \rightarrow u$ in
$L_{loc}^{2_s^\ast}(\mathbb{R}^N)$ or $\nu = |u|^{2_s^\ast} +
\sum_{j \in I} \delta_{x_j}\nu_j$, as $n \rightarrow \infty$, where
$I$ is a countable set, $\{\nu_j\}_j
\subset [0, \infty)$, $\{x_j\}_j \subset \mathbb{R}^N$.\\
\indent Take $\phi \in C_0^\infty(\mathbb{R}^N)$ such that $0 \leq
\phi \leq 1$; $\phi \equiv 1$ in $B(x_j, \rho)$, $\phi(x) = 0$ in
$\mathbb{R}^N \setminus B(x_j, 2\rho)$. For any $\rho
> 0$, define $\phi_\rho =
\phi\left(\frac{x-x_j}{\rho}\right)$, where $j \in I$. It follows
that
\begin{eqnarray}\label{e4.5}
&&\nonumber\iint_{\mathbb{R}^{2N}}\frac{|u_n(x)\phi_\rho(x)-u_n(y)\phi_\rho(y)|^2}{|x-y|^{N+2s}}dxdy\\
&&\nonumber\mbox{}\ \leq
2\iint_{\mathbb{R}^{2N}}\frac{|u_n(x)-u_n(y)|^2\phi_\rho^2(y)}{|x-y|^{N+2s}}dxdy
+
2\iint_{\mathbb{R}^{2N}}\frac{|\phi_\rho(x)-\phi_\rho(y)|^2|u_n(x)|^2}{|x-y|^{N+2s}}dxdy\\
&&\mbox{}\ \leq
2\iint_{\mathbb{R}^{2N}}\frac{|u_n(x)-u_n(y)|^2}{|x-y|^{N+2s}}dxdy
+
2\iint_{\mathbb{R}^{2N}}\frac{|\phi_\rho(x)-\phi_\rho(y)|^2|u_n(x)|^2}{|x-y|^{N+2s}}dxdy.
\end{eqnarray}
Similar to the proof of Lemma 3.4 in \cite{zhang2}, we have
\begin{eqnarray}\label{e4.6}
\iint_{\mathbb{R}^{2N}}\frac{|\phi_\rho(x)-\phi_\rho(y)|^2|u_n(x)|^2}{|x-y|^{N+2s}}dxdy
\leq C\rho^{-2s} \int_{B(x_i,K\rho)}|u_n(x)|^2dx + CK^{-N},
\end{eqnarray}
where $K > 4$.
In fact, we notice that
\begin{eqnarray*}
\mathbb{R}^{N}\times\mathbb{R}^{N} &=& ((\mathbb{R}^{N}\setminus
B(x_j,2\rho))\cup B(x_j,2\rho))\times ((\mathbb{R}^{N}\setminus
B(x_j,2\rho))\cup
B(x_j,2\rho)) \\
&=& ((\mathbb{R}^{N}\setminus B(x_j,2\rho))\times
(\mathbb{R}^{N}\setminus B(x_j,2\rho))) \cup
(B(x_j,2\rho)\times \mathbb{R}^{N}) \\
&&\mbox{} \cup ((\mathbb{R}^{N}\setminus B(x_j,2\rho))\times
B(x_j,2\rho)).
\end{eqnarray*}
Then we have
\begin{eqnarray*}
&&\iint_{\mathbb{R}^{2N}}\frac{|u_n(x)|^2|\phi_\rho(x)-\phi_\rho(y)|^2}{|x-y|^{N+2s}}dxdy
= \iint_{B(x_j,2\rho) \times
\mathbb{R}^{N}}\frac{|u_n(x)|^2|\phi_\rho(x)-\phi_\rho(y)|^2}{|x-y|^{N+2s}}dxdy
 \\
&&\mbox{}  + \iint_{(\mathbb{R}^{N}\setminus B(x_j,2\rho))\times
B(x_j,2\rho)}\frac{|u_n(x)|^2|\phi_\rho(x)-\phi_\rho(y)|^2}{|x-y|^{N+2s}}dxdy\\
&&\mbox{} \leq C\rho^{-2s}\int_{B(x_j,K\rho)}|u_n(x)|^2dx +
CK^{-N}\left(\int_{\mathbb{R}^{N}\setminus
B(x_j,K\rho)}|u_n(x)|^{2_s^\ast}dx\right)^{2/2_s^\ast} \\
&&\mbox{} \leq C\rho^{-2s}\int_{B(x_j,K\rho)}|u_n(x)|^2dx + CK^{-N}.
\end{eqnarray*}

Since $\{u_n\}_n$ is bounded in $E$,  it follows from
\eqref{e4.5} and  \eqref{e4.6} that $\{u_n\phi_\rho\}_n$ is bounded in
$E$. Then $\langle J_\varepsilon'(u_n), u_n\phi_\rho\rangle
\rightarrow 0$, which implies
\begin{eqnarray}\label{e4.7}
&&\nonumber
M\left([u_n]_{s,A_\varepsilon}^2\right)\iint_{\mathbb{R}^{2N}}\frac{|u_n(x)-e^{i(x-y)\cdot
A_\varepsilon(\frac{x+y}{2})}u_n(y)|^2\phi_\rho(y)}{|x-y|^{N+2s}}dx
dy + \varepsilon^{-2s}
\int_{\mathbb{R}^N}V(x)|u_n|^2\phi_\rho(x) dx\\
&&\mbox{}\nonumber = - \mbox{Re}\left\{
M\left([u_n]_{s,A_\varepsilon}^2\right)\iint_{\mathbb{R}^{2N}}\frac{(u_n(x)-e^{i(x-y)\cdot
A_\varepsilon(\frac{x+y}{2})}u_n(y))\overline{u_n(x)(\phi_\rho(x)-\phi_\rho(y))}}{|x-y|^{N+2s}}dx dy\right\}\\
&&\mbox{}\ \  + \varepsilon^{-2s}
\int_{\mathbb{R}^N}|u_n|^{2_s^\ast}\phi_\rho dx +
\varepsilon^{-2s}\int_{\mathbb{R}^N}h(x, |u_n|^2)|u_n|^2\phi_\rho(x)
dx +o_n(1).
\end{eqnarray}
It follows from
$\int_{\mathbb{R}^N}\frac{|u_n(x)-u_n(y)|^2}{|x-y|^{N+2s}}dy\rightharpoonup
\eta$ weakly * in $\mathcal{M}(\mathbb{R}^N)$ that
\begin{align*}
\lim_{n\rightarrow\infty}\int_{\mathbb{R}^N}\int_{\mathbb{R}^N}\frac{|u_n(x)-u_n(y)|^2\phi_\rho(y)}{|x-y|^{N+2s}}dy
dx= \int_{\mathbb{R}^N}\phi_\rho d\eta.
\end{align*}
By the diamagnetic inequality in Lemma 2.1, we have
\begin{eqnarray*}
\iint_{\mathbb{R}^{2N}}\frac{\left||u_n(x)|-|u_n(y)|\right|^2\phi_\rho(y)}{|x-y|^{N+2s}}dxdy
\leq \int_{\mathbb{R}^{N}}\phi_\rho d\eta,
\end{eqnarray*}
as $n \rightarrow \infty$ and and
\begin{eqnarray*}
\int_{\mathbb{R}^{N}}\phi_\rho d\mu \rightarrow \eta(\{x_i\})
\end{eqnarray*}
as $\rho \rightarrow 0$. Note that the  H\"{o}lder inequality
implies
\begin{eqnarray}\label{e4.8}
&&\nonumber\left|\mbox{Re}\left\{
M\left([u_n]_{s,A_\varepsilon}^2\right)\iint_{\mathbb{R}^{2N}}\frac{(u_n(x)-e^{i(x-y)\cdot
A_\varepsilon(\frac{x+y}{2})}u_n(y))\overline{u_n(x)(\phi_\rho(x)-\phi_\rho(y))}}{|x-y|^{N+2s}}dx dy\right\}\right|\\
&&\nonumber \mbox{} \ \leq
C\iint_{\mathbb{R}^{2N}}\frac{|u_n(x)-e^{i(x-y)\cdot
A_\varepsilon(\frac{x+y}{2})}u_n(y)|\cdot|\phi_\rho(x)-\phi_\rho(y)|\cdot|u_n(x)|}{|x-y|^{N+2s}}dxdy
\\
&& \mbox{} \  \leq C
\left(\iint_{\mathbb{R}^{2N}}\frac{|u_n(x)|^2|\phi_\rho(x)-\phi_\rho(y)|^2}{|x-y|^{N+2s}}dxdy\right)^{1/2}.
\end{eqnarray}
\indent  Now, we claim that
\begin{eqnarray}\label{e4.9}
\lim_{\rho\rightarrow
 0}\lim_{n\rightarrow\infty}\iint_{\mathbb{R}^{2N}}\frac{|u_n(x)|^2|\phi_\rho(x)-\phi_\rho(y)|^2}{|x-y|^{N+2s}}dxdy
=0.
\end{eqnarray}
Note that $u_n\rightharpoonup u$ weakly converges in $E$. By Proposition
\ref{pro1}  we obtain that $u_n\to u\,\
\text{in}\, \ L_{\rm{loc}}^t(\mathbb{R}^N),\ 1\leq t<2_s^\ast$, which
implies in \eqref{e4.6}
\begin{eqnarray*}
C\rho^{-2s}\int_{B(x_i,K\rho)}|u_n(x)|^2dx + CK^{-N} \rightarrow
C\rho^{-2s}\int_{B(x_i,K\rho)}|u(x)|^2dx + CK^{-N},
\end{eqnarray*}
as $n \rightarrow \infty$. Then the  H\"{o}lder inequality yields
\begin{align*}
C\rho^{-2s}\int_{B(x_i,K\rho)}|u(x)|^2dx+CK^{-N}\leq CK^{2s}\left(\int_{B(x_i,K\rho)}|u(x)|^{2_s^\ast}dx\right)^{2/2_s^\ast}+ CK^{-N}
\rightarrow  CK^{-N}
\end{align*}
as $\rho \rightarrow 0$.  Furthermore, by \eqref{e4.6} we have
\begin{eqnarray*}
&& \limsup\limits_{\rho \rightarrow 0}\limsup\limits_{n \rightarrow
\infty}
\iint_{\mathbb{R}^{2N}}\frac{|u_n(x)|^2|\phi_\rho(x)-\phi_\rho(y)|^2}{|x-y|^{N+2s}}dxdy
\\
&& \mbox{} = \lim_{K\rightarrow \infty}\limsup\limits_{\rho
\rightarrow 0}\limsup\limits_{n \rightarrow \infty}
\iint_{\mathbb{R}^{2N}}\frac{|u_n(x)|^2|\phi_\rho(x)-\phi_\rho(y)|^2}{|x-y|^{N+2s}}dxdy
=0.
\end{eqnarray*}
Hence the claim is proved.

\indent It follows from the definition of $\phi_\rho$ and $u_n\to
u\, \text{in }L_{\rm{loc}}^t(\mathbb{R}^N), 1\leq t<2_s^\ast$, that
\begin{eqnarray}\label{e4.10}
\lim_{\rho \rightarrow 0}\lim_{n \rightarrow
\infty}\int_{\mathbb{R}^N}h(x, |u_n|^2)|u_n|^2\phi_\rho(x) dx = 0,
\end{eqnarray}
and
\begin{eqnarray}\label{e4.111}
\lim_{\rho \rightarrow 0}\lim_{n \rightarrow
\infty}\int_{\mathbb{R}^N}V(x)|u_n|^2\phi_\rho(x) dx = 0.
\end{eqnarray}

Since $\phi_\rho$ has compact support, letting $n\to\infty$ in
\eqref{e4.7}, we can deduce  from \eqref{e4.8}--\eqref{e4.111} and the
diamagnetic inequality that
$$
\alpha_0\eta(\{x_j\})\leq \varepsilon^{-2s}\nu_j.
$$
 Combining this fact with Lemma \ref{lemma3.1}, we obtain
$\nu_j \geq \alpha_0\varepsilon^{2s}S \nu_j^{2/2_s^{\ast}}$.
This result implies that
$${\rm (I)} \quad \nu_j = 0 \ \indent \mbox{or}\ \quad {\rm (II)} \quad  \nu_j \geq \left(\alpha_0S\right)^{\frac{N}{2s}}\varepsilon^{N}.$$
To obtain the possible concentration of mass at infinity,
we similarly define a cutoff function $\phi_R \in C_0^\infty(\mathbb{R}^N)$
such that $\phi_R(x)=0$ on $|x| < R$ and $\phi_R(x)=1$ on $|x| >
R+1$. We can verify that $\{u_n\phi_R\}_n$ is bounded in $E$, hence
$\langle J_\varepsilon'(u_n), u_n\phi_R\rangle \rightarrow 0$,  as
$n \rightarrow \infty$, which implies
\begin{eqnarray}\label{e4.11}
&&\nonumber
M\left([u_n]_{s,A_\varepsilon}^2\right)\iint_{\mathbb{R}^{2N}}\frac{|u_n(x)-e^{i(x-y)\cdot
A_\varepsilon(\frac{x+y}{2})}u_n(y)|^2\phi_R(y)}{|x-y|^{N+2s}}dx dy
+\varepsilon^{-2s}
\int_{\mathbb{R}^N}V(x)|u_n|^2\phi_R(x) dx\\
&&\mbox{}\nonumber = - \mbox{Re}\left\{
M\left([u_n]_{s,A_\varepsilon}^2\right)\iint_{\mathbb{R}^{2N}}\frac{(u_n(x)-e^{i(x-y)\cdot
A_\varepsilon(\frac{x+y}{2})}u_n(y))\overline{u_n(x)(\phi_R(x)-\phi_R(y))}}{|x-y|^{N+2s}}dx dy\right\}\\
&&\mbox{}\ \  +
\varepsilon^{-2s}\int_{\mathbb{R}^N}|u_n|^{2_s^\ast}\phi_R dx +
\varepsilon^{-2s}\int_{\mathbb{R}^N}h(x, |u_n|^2)|u_n|^2\phi_R(x) dx
+o_n(1).
\end{eqnarray}
It is easy to verify that
\begin{eqnarray*}
\limsup\limits_{R\rightarrow\infty}\limsup\limits_{n\rightarrow\infty}\iint_{\mathbb{R}^{2N}}\frac{||u_n(x)|-|u_n(y)||^2\phi_R(y)}{|x-y|^{N+2s}}dxdy
=\eta_\infty
\end{eqnarray*}
and
\begin{eqnarray*}
&& \left|\mbox{Re}\left\{
M\left([u_n]_{s,A_\varepsilon}^2\right)\iint_{\mathbb{R}^{2N}}\frac{(u_n(x)-e^{i(x-y)\cdot
A_\varepsilon(\frac{x+y}{2})}u_n(y))\overline{u_n(x)(\phi_R(x)-\phi_R(y))}}{|x-y|^{N+2s}}dx
dy\right\}\right| \\
&& \mbox{}  \leq
C\left(\iint_{\mathbb{R}^{2N}}\frac{|u_n(x)|^2|\phi_R(x)-\phi_R(y)|^2}{|x-y|^{N+2s}}dxdy\right)^{1/2}.
\end{eqnarray*}
Note that
\begin{eqnarray*}
&& \limsup\limits_{R \rightarrow \infty}\limsup\limits_{n
\rightarrow \infty}
\iint_{\mathbb{R}^{2N}}\frac{|u_n(x)|^2|\phi_R(x)-\phi_R(y)|^2}{|x-y|^{N+2s}}dxdy
\\
&& \mbox{} = \limsup\limits_{R \rightarrow \infty}\limsup\limits_{n
\rightarrow \infty}
\iint_{\mathbb{R}^{2N}}\frac{|u_n(x)|^2|(1-\phi_R(x))-(1-\phi_R(y))|^2}{|x-y|^{N+2s}}dxdy.
\end{eqnarray*}
Similar to the proof of Lemma 3.4 in \cite{zhang2}, we have
\begin{eqnarray}\label{e4.12}
\limsup\limits_{R \rightarrow \infty}\limsup\limits_{n \rightarrow
\infty}
\iint_{\mathbb{R}^{2N}}\frac{|u_n(x)|^2|(1-\phi_R(x))-(1-\phi_R(y))|^2}{|x-y|^{N+2s}}dxdy
= 0.
\end{eqnarray}
It follows from the definition of $\phi_R$ that
\begin{eqnarray}\label{e4.13}
\lim_{R \rightarrow \infty}\lim_{n \rightarrow
\infty}\int_{\mathbb{R}^N}h(x, |u_n|^2)|u_n|^2\phi_R(x) dx = 0
\end{eqnarray}
and
\begin{eqnarray}\label{e4.14}
\lim_{R \rightarrow \infty}\lim_{n \rightarrow
\infty}\int_{\mathbb{R}^N}V(x)|u_n|^2\phi_R(x) dx = 0.\end{eqnarray}
It follows from  \eqref{e4.12}--\eqref{e4.14} that
\begin{eqnarray*}
\alpha_0\mu_\infty\leq  \varepsilon^{-2s}\nu_\infty
\end{eqnarray*}
as $R \to \infty$ in \eqref{e4.11}.
By Lemma \ref{lemma3.2}, we obtain $\nu_\infty \geq\alpha_0
\varepsilon^{2s}S
 \nu_\infty^{2/2_s^{\ast}}$. This result implies that
$${\rm (III)} \quad \nu_\infty = 0 \ \indent \mbox{or}\ \quad {\rm (IV)} \quad  \nu_\infty \geq \left(\alpha_0S\right)^{\frac{N}{2s}}\varepsilon^{N}.$$
\indent Next, we claim that $(II)$ and $(IV)$ cannot occur. If the
case $(IV)$ holds for some $j \in I$, then by Lemma
\ref{lemma3.2}, $(M)$ and $(H)$, we have
\begin{eqnarray*}
c &=& \lim_{n \rightarrow \infty}\left(J_\varepsilon(u_n) -
\frac{1}{\mu}\langle
J'_\varepsilon(u_n), u_n\rangle\right)\\
&\geq&  \left(\frac{\sigma}{2}-\frac{1}{\mu}\right)
M\left([u_n]_{s,A_\varepsilon}^2\right)[u_n]_{s,A_\varepsilon}^2
+\left(\frac{1}{2}-\frac{1}{\mu}\right)\varepsilon^{-2s}
\int_{\mathbb{R}^N}V(x)|u_n|^2dx   \\
 &&\nonumber \mbox{} +
\left(\frac{1}{\mu}-\frac{1}{2_s^\ast}\right)\varepsilon^{-2s}\int_{\mathbb{R}^N} |u_n|^{2_s^\ast}dx + \varepsilon^{-2s} \int_{\mathbb{R}^N} \left[\frac{1}{\mu}h(x, |u_n|^2)|u_n|^2-\frac{1}{2}H(x,|u_n|^2)\right]dx\\
&\geq&
\left(\frac{1}{\mu}-\frac{1}{2_s^\ast}\right)\varepsilon^{-2s}\int_{\mathbb{R}^N}
|u_n|^{2_s^\ast}dx
 \geq
\left(\frac{1}{\mu}-\frac{1}{2_s^\ast}\right)\varepsilon^{-2s}\nu_\infty
\geq \sigma_0\varepsilon^{N-2s},
\end{eqnarray*}
where $\sigma_0 =
\left(\frac{1}{\mu}-\frac{1}{2_s^\ast}\right)(\alpha_0S)^{N/(2s)}$,
which contradicts the condition $c \in \left(0,\,
 \sigma_0\varepsilon^{N-2s}\right)$. Consequently, $\nu_j = 0$ for all $j\in I$.
Similarly, we can prove that $(II)$ cannot occur for any $j$. Thus
\begin{eqnarray} \label{e3.11}
\int_{\mathbb{R}^N}|u_n|^{2_s^\ast}dx \rightarrow
\int_{\mathbb{R}^N}|u|^{2_s^\ast}dx,
\end{eqnarray}
as $n\rightarrow \infty$. Since $|u_n - u|^{2_s^\ast} \leq
{2^{2_s^\ast}}(|u_n|^{{2_s^\ast}}+|u|^{2_s^\ast})$, it follows from
the Fatou lemma that
\begin{eqnarray*}
\int_{\mathbb{R}^N}2^{2_s^\ast+1}|u|^{2^{2_s^\ast}} dx &=&
\int_{\mathbb{R}^N}\liminf\limits_{n \rightarrow
\infty}(2^{2_s^\ast}|u_n|^{2^{2_s^\ast}}+2^{2_s^\ast}|u|^{2^{2_s^\ast}}-|u_n-u|^{2^{2_s^\ast}})dx
\\
&\leq& \liminf\limits_{n \rightarrow
\infty}\int_{\mathbb{R}^N}(2^{2_s^\ast}|u_n|^{2^{2_s^\ast}}+2^{2_s^\ast}|u|^{2^{2_s^\ast}}-|u_n-u|^{2^{2_s^\ast}})dx
\\
&=& \int_{\mathbb{R}^N}2^{2_s^\ast+1}|u|^{2^{2_s^\ast}} dx -
\limsup\limits_{n \rightarrow
\infty}\int_{\mathbb{R}^N}|u_n-u|^{2^{2_s^\ast}}dx,
\end{eqnarray*}
which implies that $\limsup\limits_{n \rightarrow
\infty}\int_{\mathbb{R}^N}|u_n-u|^{2_s^\ast}dx = 0$. Then
\begin{eqnarray*}
u_n \rightarrow u \quad \mbox{in}\quad L^{2_s^\ast}(\mathbb{R}^N)
\quad \mbox{as}\quad n \rightarrow \infty.
\end{eqnarray*}
By the weak lower semicontinuity of the norm,
conditon $(m_1)$ and the Br\'{e}zis-Lieb lemma, we have
\begin{eqnarray*}
 o(1)\|u_n\|&=& \nonumber\langle J_\varepsilon'(u_n), u_n\rangle =  M\left([u_n]_{s,A_\varepsilon}^2\right)[u_n]_{s,A_\varepsilon}^2
+ \varepsilon^{-2s}\int_{\mathbb{R}^N} V(x)|u_n|^2dx \\
&&\mbox{} - \varepsilon^{-2s}\int_{\mathbb{R}^N}|u_n|^{2^\ast(s)}dx
- \varepsilon^{-2s}\int_{\mathbb{R}^N}h(x, |u_n|^2)|u_n|^2dx\\
&\geq& \alpha_0\left([u_n]_{s,A_\varepsilon}^2 -
[u]_{s,A_\varepsilon}^2\right)
+ \varepsilon^{-2s}\int_{\mathbb{R}^N} V(x)(|u_n|^2-|u|^2)dx + M\left([u]_{s,A_\varepsilon}^2\right)[u]_{s,A_\varepsilon}^2\\
&&\mbox{}+ \varepsilon^{-2s}\int_{\mathbb{R}^N} V(x)|u|^2dx -
 \varepsilon^{-2s}\int_{\mathbb{R}^N}|u|^{2_s^\ast}dx -
\varepsilon^{-2s}\int_{\mathbb{R}^N}h(x, |u|^2)|u|^2dx\\
&\geq& \min\{\alpha_0,1\}\|u_n - u\|_\varepsilon^2  +
o(1)\|u\|_\varepsilon.
\end{eqnarray*}
Here we use the fact that $J_\varepsilon'(u) = 0$. Thus we have proved that $\{u_n\}_n$
strongly converges to $u$ in $E$. Hence the proof is complete.
\end{proof}

%%%%%%%%%%%%%%%%%%%%%%%%%%%%%%%%%%%%%%%%%%%%%%%%%%%%%%%%%%%%%%%%%%%%%%%%%%%%%%%%%%%
\section{Proof of Theorem \ref{the3.1} }

\indent In the following, we will always consider $0 < \varepsilon < 1$.
By the assumptions $(V)$, $(M)$ and $(H)$, one can see that
$J_\varepsilon(u)$ has the mountain pass geometry.
\begin{lemma}\label{lemma4.1} Assume that conditions {\rm ($V$)}, $(M)$ and  {\rm ($H$)}  hold. Then there exist $\alpha_\varepsilon, \varrho_\varepsilon
> 0$ such that $J_\varepsilon(u) > 0$ if $u \in B_{\varrho_\varepsilon}\setminus\{0\}$ and $J_\varepsilon(u) \geq \alpha_\varepsilon$
if $u \in \partial B_{\varrho_\varepsilon}$, where
$B_{\varrho_\varepsilon} = \{u\in E: \|u\|_\varepsilon \leq
\varrho_\varepsilon\}$.\end{lemma}
\begin{proof} By $(H)$, for $0<\xi
\leq \left(2\min\left\{\frac{\sigma\alpha_0}{2},\frac{1}{2}\right
\}c_2^2\right)^{-1}\varepsilon^{2s} $, there is $C_\xi
> 0$ such that
\begin{displaymath}
\frac{1}{2_s^\ast}\int_{\mathbb{R}^N}|u|^{2_s^\ast}dx +
\frac{1}{2}\int_{\mathbb{R}^N}H(x, |u|^2)dx \leq \xi \|u\|_{L^2}^2 +
C_\xi \|u\|_{L^{2_s^\ast}}^{2_s^\ast},
\end{displaymath}
where $c_2$ is the embedding constant in \eqref{e2.2} with $\theta=2$.
it follows from $(V)$, $(M)$ and $(H)$, that
\begin{eqnarray*}
J_\varepsilon(u) &=& \frac12 \widetilde{M}\left([u]_{s,A_\varepsilon}^2\right)+
\frac{\varepsilon^{-2s} }{2}\int_{\mathbb{R}^N}V(x)|u|^2dx -
\frac{\varepsilon^{-2s}}{2_s^\ast}\int_{\mathbb{R}^N}
|u|^{2_s^\ast}dx
  - \frac{\varepsilon^{-2s}}{2}\int_{\mathbb{R}^N}H(x, |u|^2)dx\\
 &\geq& \min\left\{\frac{\sigma\alpha_0}{2},\frac{1}{2}\right
\}\|u\|_\varepsilon^2 - \varepsilon^{-2s}\xi \|u\|_{L^2}^2 -\varepsilon^{-2s} C_\xi  \|u\|_{L^{2_s^\ast}}^{2_s^\ast}\\
  &\geq&\frac12\min\left\{\frac{\sigma\alpha_0}{2},\frac{1}{2}\right
\}\|u\|_\varepsilon^2   -\varepsilon^{-2s} C_\xi  \|u\|_{L^{2_s^\ast}}^{2_s^\ast}\\
 &\geq&\frac12\min\left\{\frac{\sigma\alpha_0}{2},\frac{1}{2}\right
\}\|u\|_\varepsilon^2 - \varepsilon^{-2s}
 C_\xi c_{2_s^\ast}^{2_s^\ast}
 \|u\|_\varepsilon^{2_s^\ast}.
\end{eqnarray*}
Then, for all $u\in E$, with $\|u\|_\varepsilon=\rho_\varepsilon, \rho_\varepsilon\in (0,1)$ sufficiently small so that
$$\frac12\min\left\{\frac{\sigma\alpha_0}{2},\frac{1}{2}\right
\} - \varepsilon^{-2s}
 C_\xi c_{2_s^\ast}^{2_s^\ast}
\rho_\varepsilon^{2_s^\ast-2}>0.$$
Thus, the lemma is proved by taking
$$\alpha_\varepsilon=\frac12\min\left\{\frac{\sigma\alpha_0}{2},\frac{1}{2}\right
\}\rho_\varepsilon^{2} - \varepsilon^{-2s}
 C_\xi c_{2_s^\ast}^{2_s^\ast}
\rho_\varepsilon^{2_s^\ast}.$$
The proof is finished.
\end{proof}
\begin{lemma}\label{lemma4.2}  Under the assumptions of Lemma \ref{lemma4.1},
for any finite dimensional subspace $F \subset E$,
\begin{displaymath}
J_\varepsilon(u) \rightarrow -\infty\quad \mbox{as}\quad
\|u\|_\varepsilon \rightarrow \infty\ \mbox{with} \ u\in F.
\end{displaymath}
\end{lemma}
\begin{proof}  By integrating {\rm $(m_2)$}, we obtain
\begin{equation}\label{e5.1}
\widetilde{M}(t)\leq \frac{\widetilde{M}(t_0)}{t_0^{1/\sigma}}t^{1/\sigma}
=C_0t^{1/\sigma}\quad \text{for all }t\geq t_0>0.
\end{equation}
Using conditions $(V)$ and $(H)$, we can get
\begin{displaymath}
J_\varepsilon(u) \leq
\frac{C_0}{2}\|u\|_\varepsilon^{\frac{2}{\sigma}} +
\frac{1}{2}\|u\|_\varepsilon^{2}  -
\frac{\varepsilon^{-2s}}{2_s^\ast}\|u\|_{L^{2_s^\ast}}^{2_s^\ast} -
\varepsilon^{-2s} l_0 \|u\|_{L^r}^r
\end{displaymath}
for all $u \in F$. Since all norms in a finite-dimensional space are
equivalent and   $2\leq2/\sigma < 2_s^\ast$, we conclude that
$J_\varepsilon(u) \rightarrow -\infty$ as $\|u\|_\varepsilon
\rightarrow \infty$. The proof is thus complete.\end{proof} \indent
Note that $J_\varepsilon(u)$ does not satisfy $(PS)_c$ condition for
any $c > 0$. Thus, in the sequel we will find a special
finite-dimensional subspace by which we will construct sufficiently
small minimax levels. \\
\indent Recall that the assumption $(V)$ implies that there is $x_0 \in
\mathbb{R}^N$ such that $V(x_0) = \min_{x\in \mathbb{R}^N} V(x) =
0$. Without loss of
generality we can assume from now on that $x_0 = 0$.
We first notice that condition $(h_3)$ implies
\begin{eqnarray*}
 \frac{\varepsilon^{-2s}}{2_s^\ast}\int_{\mathbb{R}^N}|u|^{2_s^\ast}dx +
 \frac{\varepsilon^{-2s}}{2}\int_{\mathbb{R}^N}H(x, |u|^2)dx \geq
l_0\varepsilon^{-2s}\int_{\mathbb{R}^N} |u|^rdx.
\end{eqnarray*}
Define the functional $I_\varepsilon \in C^1(E, \mathbb{R})$ by
\begin{displaymath}
I_\varepsilon(u) := \frac12 M\left([u]_{s,A_\varepsilon}^2\right)+
\frac{\varepsilon^{-2s}}{2}\int_{\mathbb{R}^N} V(x)|u|^2dx -
l_0\varepsilon^{-2s}\int_{\mathbb{R}^N} |u|^rdx.
\end{displaymath}
Then $J_\varepsilon(u) \leq I_\varepsilon(u)$  for all $u \in E$ and
it suffices to construct small minimax levels for $I_\varepsilon$.\\
\indent Note that
\begin{displaymath}
\inf\left\{\iint_{\mathbb{R}^{2N}}\frac{|\phi(x)-\phi(y)|^2}{|x-y|^{N+2s}}dxdy:
\phi \in C_0^\infty (\mathbb{R}^N), |\phi|_q = 1\right\} = 0,
\end{displaymath}
see \cite[Theorem 3.2]{zhang3} for this proof. For any $0< \zeta <1$ one can choose $\phi_\zeta \in C_0^\infty
(\mathbb{R}^N)$ with $\|\phi_\zeta\|_{L^q }= 1$ and supp\,$\phi_\zeta
\subset B_{r_\zeta} (0)$ so that
\begin{displaymath}
\iint_{\mathbb{R}^{2N}}\frac{|\phi_\zeta(x)-\phi_\zeta(y)|^2}{|x-y|^{N+2s}}dxdy
\leq C\zeta^{\frac{2N-(N-2s)q}{q}}.
\end{displaymath}
Set
\begin{equation}\label{e5.2}
\psi_\zeta(x) = e^{iA(0)x}\phi_\zeta(x)
\end{equation}
and
\begin{equation}\label{e5.3}
\psi_{\varepsilon,\zeta}(x) = \psi_\zeta(\varepsilon^{-1}x).
\end{equation}
By condition \eqref{e5.1}, we get for any $t>0$
\begin{eqnarray*}
I_\varepsilon(t\psi_{\varepsilon,\zeta})
&\leq&\frac{C_0}{2}t^{\frac{2}{\sigma}}\left(\iint_{\mathbb{R}^{2N}}\frac{|\psi_{\varepsilon,\zeta}(x)-e^{i(x-y)\cdot
A_\varepsilon(\frac{x+y}{2})}\psi_{\varepsilon,\zeta}(y)|^2}{|x-y|^{N+2s}}dx dy\right)^{1/\sigma} \\
&& \mbox{} + \frac{t^2}{2}\varepsilon^{-2s}
\int_{\mathbb{R}^N}V(x)|\psi_{\varepsilon,\zeta}|^2dx  - t^rl_0\varepsilon^{-2s}\int_{\mathbb{R}^N} |\psi_{\varepsilon,\zeta}|^rdx\\
&\leq&
\varepsilon^{N-2s}\left[\frac{C_0}{2}t^{\frac{2}{\sigma}}\left(\iint_{\mathbb{R}^{2N}}\frac{|\psi_{\zeta}(x)-e^{i(x-y)\cdot
A(\frac{\varepsilon x+\varepsilon y}{2})}\psi_{\zeta}(y)|^2}{|x-y|^{N+2s}}dx dy\right)^{1/\sigma} \right.\\
&& \mbox{} \left. +
\frac{t^2}{2}\int_{\mathbb{R}^N}V\left(\varepsilon
x\right)|\psi_{\zeta}|^2dx-
t^rl_0\int_{\mathbb{R}^N} |\psi_{\zeta}|^rdx\right]\\
&=& \varepsilon^{N-2s} \Psi_\varepsilon(t\psi_{\zeta}),
\end{eqnarray*}
where $\Psi_\varepsilon \in C^1(E, \mathbb{R})$ defined by
\begin{eqnarray*}
\Psi_\varepsilon(u) &:=&
\frac{C_0}{2}\left(\iint_{\mathbb{R}^{2N}}\frac{|u(x)-e^{i(x-y)\cdot
A(\frac{\varepsilon x+\varepsilon y}{2})}u(y)|^2}{|x-y|^{N+2s}}dx
dy\right)^{1/\sigma} \\
&& \mbox{} + \frac{1}{2}\int_{\mathbb{R}^N}V\left(\varepsilon
x\right)|u|^2dx - l_0\int_{\mathbb{R}^N} |u|^rdx.
\end{eqnarray*}
Since $r > 2/\sigma$,  there exists a finite number $t_0
\in [0, +\infty)$ such that
\begin{eqnarray*}
\max_{t \geq 0} \Psi_\varepsilon(t\psi_{\zeta}) &=&
\frac{C_0}{2}t_0^{2/\sigma}\left(\iint_{\mathbb{R}^{2N}}\frac{|\psi_{\zeta}(x)-e^{i(x-y)\cdot
A(\frac{\varepsilon x+\varepsilon
y}{2})}\psi_{\zeta}(y)|^2}{|x-y|^{N+2s}}dx
dy\right)^{1/\sigma}  \\
&& \mbox{} + \frac{t_0^2}{2}\int_{\mathbb{R}^N}V\left(\varepsilon
x\right)|\psi_{\zeta}|^2dx-
t_0^rl_0\int_{\mathbb{R}^N} |\psi_{\zeta}|^rdx\\
&\leq&
\frac{C_0}{2}t_0^{2/\sigma}\left(\iint_{\mathbb{R}^{2N}}\frac{|\psi_{\zeta}(x)-e^{i(x-y)\cdot
A(\frac{\varepsilon x+\varepsilon
y}{2})}\psi_{\zeta}(y)|^2}{|x-y|^{N+2s}}dx
dy\right)^{1/\sigma} \\
&& \mbox{} + \frac{t_0^2}{2}\int_{\mathbb{R}^N}V\left(\varepsilon
x\right)|\psi_\zeta|^2dx.
\end{eqnarray*}
Let $\psi_\zeta(x) = e^{iA(0)x}\phi_\zeta(x)$, where $\phi_\zeta(x)$
is as defined above. We have the following lemma.
\begin{lemma}\label{lemma5.3}   For any $\zeta > 0$ there exists $\varepsilon_0 = \varepsilon_0(\zeta) >
0$ such that
\begin{eqnarray*}
\iint_{\mathbb{R}^{2N}}\frac{|\psi_{\zeta}(x)-e^{i(x-y)\cdot
A(\frac{\varepsilon x+\varepsilon
y}{2})}\psi_{\zeta}(y)|^2}{|x-y|^{N+2s}}dx dy \leq
C\zeta^{\frac{2N-(N-2s)q}{q}} + \frac{1}{1-s}\zeta^{2s} +
\frac{4}{s}\zeta^{2s},
\end{eqnarray*}
for all $0 < \varepsilon < \varepsilon_0$ and some constant $C > 0$
depending only on $[\phi_\zeta]_{s,0}$.
\end{lemma}
\begin{proof}
For any $\zeta > 0$, we have
\begin{eqnarray*}
&& \iint_{\mathbb{R}^{2N}}\frac{|\psi_{\zeta}(x)-e^{i(x-y)\cdot
A(\frac{\varepsilon x+\varepsilon
y}{2})}\psi_{\zeta}(y)|^2}{|x-y|^{N+2s}}dx
dy \\
&& \leq \iint_{\mathbb{R}^{2N}}\frac{|e^{iA(0)\cdot
x}\phi_{\zeta}(x)-e^{i(x-y)\cdot A(\frac{\varepsilon x+ \varepsilon
y}{2})}e^{iA(0)\cdot y}\phi_{\zeta}(y)|^2}{|x-y|^{N+2s}}dx
dy\\
&& \leq
2\iint_{\mathbb{R}^{2N}}\frac{|\phi_{\zeta}(x)-\phi_{\zeta}(y)|^2}{|x-y|^{N+2s}}dx
dy +
2\iint_{\mathbb{R}^{2N}}\frac{|\phi_{\zeta}(y)|^2|e^{i(x-y)\cdot(A(0)-
A(\frac{\varepsilon x+\varepsilon y}{2}))}-1|^2}{|x-y|^{N+2s}}dx dy.
\end{eqnarray*}
Next, we will estimate the second term in the above inequality.
Notice that
\begin{equation}\label{e4.22}
\left|e^{i(x-y)\cdot(A(0)- A(\frac{\varepsilon x+\varepsilon
y}{2}))}-1\right|^2 =
4\sin^2\left[\frac{(x-y)\cdot(A(0)-A(\frac{\varepsilon x+\varepsilon
y}{2}))}{2}\right].
\end{equation}
For any $y \in B_{r_\zeta}$, if $|x-y| \leq
\frac{1}{\zeta}\|\phi_\zeta\|_{L^2}^{1/\alpha}$, then $|x|
\leq r_\zeta +
\frac{1}{\zeta}\|\phi_\zeta\|_{L^2}^{1/\alpha}$. Hence, we
have
\begin{eqnarray*}
\left|\frac{\varepsilon x+\varepsilon y}{2}\right| \leq
\frac{\varepsilon}{2} \left(2r_\zeta +
\frac{1}{\zeta}\|\phi_\zeta\|_{L^2}^{1/\alpha}\right).
\end{eqnarray*}
Since $A: \mathbb{R}^{N} \rightarrow \mathbb{R}^{N}$ is continuous,
there exists $\varepsilon_0 > 0$ such that for any $0 < \varepsilon
< \varepsilon_0$, we have
\begin{eqnarray*}
\left|A(0)-A\left(\frac{\varepsilon x+\varepsilon
y}{2}\right)\right| \leq \zeta
\|\phi_\zeta\|_{L^2}^{-1/\alpha} \quad \mbox{for} \ |y| \leq
r_\zeta \ \mbox{}and \ |x| \leq r_\zeta +
\frac{1}{\zeta}\|\phi_\zeta\|_{L^2}^{1/\alpha},
\end{eqnarray*}
which implies
\begin{eqnarray*}
\left|e^{i(x-y)\cdot(A(0)- A(\frac{\varepsilon x+\varepsilon
y}{2}))}-1\right|^2  \leq
|x-y|^2\zeta^2\|\phi_\zeta\|_{L^2}^{-2/\alpha}.
\end{eqnarray*}
For all $\zeta > 0$ and $y \in B_{r_\zeta}$, let us define
\begin{eqnarray*}
N_{\zeta,y} := \left\{x \in \mathbb{R}^{N}: |x-y| \leq
\frac{1}{\zeta}\|\phi_\zeta\|_{L^2}^{1/\alpha}\right\}.
\end{eqnarray*}
Then together with the above facts, we have for all $0 < \varepsilon <
\varepsilon_0$
\begin{eqnarray*}
&&\iint_{\mathbb{R}^{2N}}\frac{|\phi_{\zeta}(y)|^2|e^{i(x-y)\cdot(A(0)-
A(\frac{\varepsilon x+\varepsilon y}{2}))}-1|^2}{|x-y|^{N+2s}}dx dy \\
&& =\int_{B_{r_\zeta}}|\phi_\zeta(y)|^2dy\int_{N_{\zeta,y}}\frac{\left|e^{i(x-y)\cdot(A(0)-
A(\frac{\varepsilon x+\varepsilon
y}{2}))}-1\right|^2}{|x-y|^{N+2s}}dx +\\
&&\ \ \ \  \ \ \ \ \ \ \ \  \ \ \ \ \ \ \ \ \ \ \ \ \ \ \ \ \ \ \ \ \ \ \ \ \ \ \ \ \ \ \ \int_{B_{r_\zeta}}|\phi_\zeta(y)|^2dy\int_{\mathbb{R}^{N} \setminus
N_{\zeta,y}}\frac{\left|e^{i(x-y)\cdot(A(0)-
A(\frac{\varepsilon x+\varepsilon
y}{2}))}-1\right|^2}{|x-y|^{N+2s}}dx\\
&&\leq \int_{B_{r_\zeta}}|\phi_\zeta(y)|^2dy\int_{N_{\zeta,y}}
\frac{|x-y|^2}{|x-y|^{N+2s}}\zeta^2|\phi_\zeta|_{L^2}^{-\frac{2}{\alpha}}dx
+ \int_{B_{r_\zeta}}|\phi_\zeta(y)|^2dy\int_{\mathbb{R}^{N}
\setminus N_{\zeta,y}}\frac{4}{|x-y|^{N+2s}}dx \\
&& \leq \frac{1}{2-2s}\zeta^{2s} + \frac{4}{2s}\zeta^{2s}.
\end{eqnarray*}
This completes the proof.\end{proof}
Next, since $V(0) = 0$ and supp\,$\phi_\zeta \subset
B_{r_\zeta}(0)$, there is $\varepsilon^\ast > 0$ such that
\begin{displaymath}
V\left(\varepsilon x\right) \leq \frac{\zeta}{\|\phi_\zeta\|_{L^2}^2}\quad
\mbox{for\ all }\ |x| \leq r_\zeta\ \mbox{and}\ 0 < \varepsilon <
\varepsilon^\ast.
\end{displaymath}
This together with Lemma \ref{lemma5.3} implies that
\begin{equation}\label{e5.5}
\max_{t\geq 0}\Psi_\varepsilon(t\phi_\zeta) \leq
\frac{C_0}{2}t_0^{2/\sigma}\left(C\zeta^{\frac{2N-(N-2s)q}{q}}
+ \frac{1}{1-s}\zeta^{2s} + \frac{4}{s}\zeta^{2s}\right)^{1/\sigma}
+ \frac{t_0^2}{2}\zeta.
\end{equation}
Therefore, we have for all $0 < \varepsilon <
\min\{\varepsilon_0,\varepsilon^\ast\}$,
\begin{equation}\label{e5.6}
\max_{t\geq 0} J_\varepsilon(t\psi_{\varepsilon,\zeta}) \leq
\left[\frac{C_0}{2}t_0^{2/\sigma}\left(C\zeta^{\frac{2N-(N-2s)q}{q}}
+ \frac{1}{1-s}\zeta^{2s} + \frac{4}{s}\zeta^{2s}\right)^{1/\sigma}
+ \frac{t_0^2}{2}\zeta\right]\varepsilon^{N-2s}.
\end{equation}
\indent We are now ready to prove the following lemma.
\begin{lemma}\label{lemma4.3} Under the assumptions of Lemma \ref{lemma4.1},
for any $\kappa > 0$ there exists $\mathcal {E}_\kappa > 0$ such
that for each $0 < \varepsilon < \mathcal {E}_\kappa$, there is
$\widehat{e}_\varepsilon \in E$ with $\|\widehat{e}_\varepsilon\|_{\varepsilon} >
\varrho_\varepsilon$, $J_\varepsilon(\widehat{e}_\varepsilon) \leq
0$, and
\begin{equation}\label{e5.7}
\max_{t\in [0, 1]} J_\varepsilon(t\widehat{e}_\varepsilon) \leq
\kappa\varepsilon^{N-2s}.
\end{equation}
\end{lemma}
\begin{proof}
Choose $\zeta > 0$ so small that
\begin{displaymath}
\frac{C_0}{2}t_0^{\frac{2}{\sigma}}\left(C\zeta^{\frac{2N-(N-2s)q}{q}}
+ \frac{1}{1-s}\zeta^{2s} + \frac{4}{s}\zeta^{2s}\right)^{\frac{1}{\sigma}}
+ \frac{1}{2}t_{0}^{2}\zeta \leq \kappa.
\end{displaymath}
Let $\psi_{\varepsilon,\zeta} \in E$ be the function defined by
\eqref{e5.3}. Set $\mathcal {E}_\kappa =
\min\{\varepsilon_0,\varepsilon^\ast\}$. Let
$\widehat{t}_\varepsilon > 0$ be such that
$\widehat{t}_\varepsilon\|\psi_{\varepsilon,\zeta}\|_\varepsilon
> \varrho_\varepsilon$ and $J_\varepsilon(t\psi_{\varepsilon,\zeta}) \leq 0$ for
all $t \geq \widehat{t}_\varepsilon$. Invoking \eqref{e5.6}, we let
$\widehat{e}_\varepsilon = \widehat{t}_\varepsilon
\psi_{\varepsilon,\zeta}$ and check that the conclusion of Lemma
\ref{lemma4.3} holds.
\end{proof}
\indent For any $m^{\ast} \in \mathbb{N}$, one can choose $m^{\ast}$
functions $\phi_\zeta^i \in C_0^\infty(\mathbb{R}^N)$ such that
supp\,$\phi_\zeta^i$ $ \cap$ supp\,$\phi_\zeta^k = \emptyset$, $i
\neq k$, $\|\phi_\zeta^i\|_{L^s} = 1$ and
\begin{displaymath}
\iint_{\mathbb{R}^{2N}}\frac{|\phi_\zeta^i(x)-\phi_\zeta^i(y)|^2}{|x-y|^{N+2s}}dxdy
\leq C\zeta^{\frac{2N-(N-2s)q}{q}}.
\end{displaymath}
Let $r_\zeta^{m^{\ast}}
> 0$ be such that supp\,$\phi_\zeta^{i} \subset B_{r_\zeta}^{i}(0)$
for $i = 1,2,\cdots,m^{\ast}$. Set
\begin{equation}\label{e5.8}
\psi_\zeta^i(x) = e^{iA(0)x}\phi_\zeta^i(x)
\end{equation}
and
\begin{equation}\label{e5.9}
\psi_{\varepsilon,\zeta}^i(x) = \psi_\zeta^i(\varepsilon^{-1}x).
\end{equation}
Denote
\begin{displaymath}
H_{\varepsilon\zeta}^{m^{\ast}} =
\mbox{span}\{\psi_{\varepsilon,\zeta}^1, \psi_{\varepsilon,\zeta}^2,
\cdots, \psi_{\varepsilon,\zeta}^{m^{\ast}}\}.
\end{displaymath}
Observe that for each $u = \displaystyle\sum_{i=1}^{m^{\ast}}c_i
\psi_{\varepsilon,\zeta}^i \in H_{\varepsilon\zeta}^{m^{\ast}}$, we
have
\begin{displaymath}
[u]_{s,A_\varepsilon}^2 \leq
C\displaystyle\sum_{i=1}^{m^{\ast}}|c_i|^2[\psi_{\varepsilon,\zeta}^i]_{s,A_\varepsilon}^2,
\end{displaymath}
for some constant $C > 0$,\\
\begin{displaymath}
\int_{\mathbb{R}^N}V(x)|u|^2dx =
\displaystyle\sum_{i=1}^{m^{\ast}}|c_i|^2\int_{\mathbb{R}^N}V(x)|\psi_{\varepsilon,\zeta}^i|^2dx
\end{displaymath}
and
\begin{displaymath}
\frac{1}{2_s^\ast}\int_{\mathbb{R}^N}|u|^{2_s^\ast}dx +
\frac{1}{2}\int_{\mathbb{R}^N}H(x, |u|^2)dx =
\displaystyle\sum_{i=1}^{m^{\ast}}\left(\frac{1}{2_s^\ast}\int_{\mathbb{R}^N}|c_i
\psi_{\varepsilon,\zeta}^i|^{2_s^\ast}dx +
\frac{1}{2}\int_{\mathbb{R}^N}H(x, |c_i
\psi_{\varepsilon,\zeta}^i|^2)dx\right).
\end{displaymath}
Therefore
\begin{displaymath}
J_\varepsilon(u) \leq  C\sum_{i=1}^{m^{\ast}}J_{\varepsilon}(c_i
\psi_{\varepsilon,\zeta}^i)
\end{displaymath}
for some constant $C > 0$. By a similar argument as before, we
can see that
\begin{displaymath}
J_\varepsilon(c_i \psi_{\varepsilon,\zeta}^i) \leq \varepsilon^{N -
2s}\Psi(|c_i|\psi_{\zeta}^i).
\end{displaymath}
As before, we can obtain the following estimate:
\begin{equation}\label{e5.10}
\max_{u\in H_{\varepsilon\zeta}^{m^{\ast}}} J_\varepsilon(u) \leq C
m^\ast
\left[\frac{C_0}{2}t_0^{2/\sigma}\left(C\zeta^{\frac{2N-(N-2s)q}{q}}
+ \frac{1}{1-s}\zeta^{2s} + \frac{4}{s}\zeta^{2s}\right)^{1/\sigma}
+ \frac{t_0^2}{2}\zeta\right]\varepsilon^{N-2s}
\end{equation}
for all small enough $\zeta$ and some constant $C > 0$. Using the estimate \eqref{e5.10} we shall prove the following lemma.
\begin{lemma}\label{lemma4.4} Under the assumptions of Lemma \ref{lemma4.1},
for any $m^{\ast} \in \mathbb{N}$ and $\kappa > 0$ there exists
$\mathcal {E}_{m^{\ast}\kappa} > 0$ such that for each $0 <
\varepsilon < \mathcal {E}_{m^{\ast}\kappa}$, there exists an
$m^{\ast}$-dimensional subspace $F_{\varepsilon m^{\ast}}$ satisfying
\begin{displaymath}
\max_{u\in F_{\varepsilon m^{\ast}}} J_\varepsilon(u) \leq
\kappa\varepsilon^{N-2s}.
\end{displaymath}
\end{lemma}
\begin{proof}
Choose $\zeta > 0$ so small that
\begin{displaymath}
 C m^\ast
\left[\frac{C_0}{2}t_0^{\frac{2}{\sigma}}\left(C\zeta^{\frac{2N-(N-2s)q}{q}}
+ \frac{1}{1-s}\zeta^{2s} + \frac{4}{s}\zeta^{2s}\right)^{1/\sigma}
+ \frac{t_0^2}{2}\zeta\right] \leq \kappa.
\end{displaymath}
Set $F_{\varepsilon m^{\ast}} =
H_{\varepsilon\zeta}^{m^{\ast}}=\mbox{span}\{\psi_{\varepsilon,\zeta}^1,
\psi_{\varepsilon,\zeta}^2, \cdots,
\psi_{\varepsilon,\zeta}^{m^{\ast}}\}$. Now the conclusion of Lemma \ref{lemma4.4} follows from \eqref{e5.10}.
\end{proof}
\indent We are now ready to prove our main result which establishes the existence and multiplicity
of solutions.

\vspace{2mm}

\noindent{\bf Proof of Theorem 3.1 (1).} For any $0 < \kappa <
\sigma_0$, by Theorem \ref{lemma3.4}, we can choose $\mathcal {E}_\kappa >
0$ and define for $0 < \varepsilon < \mathcal {E}_\kappa$, the
minimax value
$$c_\varepsilon := \inf_{\gamma \in \Gamma_\varepsilon}\max_{t\in [0,1]} J_\varepsilon(t\widehat{e}_\varepsilon),$$
where $$\Gamma_\varepsilon := \{\gamma \in C([0, 1], E): \gamma(0) =
0 \ \mbox{and}\ \gamma(1) = \widehat{e}_\varepsilon\}.$$ By Lemma
\ref{lemma4.1}, we have $\alpha_\varepsilon \leq c_\varepsilon \leq
\kappa\varepsilon^{N-2s}$. By virtue of Theorem \ref{lemma3.4}, we
know that $J_\varepsilon$ satisfies the $(PS)_{c_\lambda}$
condition. In view of Lemmas 5.1 and 5.4,  it follows from the mountain
pass theorem that there is $u_\varepsilon \in E$ such that
$J'_\varepsilon(u_\varepsilon) = 0$ and
$J_\varepsilon(u_\varepsilon) = c_\varepsilon$, then $u_\varepsilon$
is a nontrivial mountain
pass solution of problem  \eqref{e3.1}.

 Since $u_\varepsilon$ is a critical point of
$J_\varepsilon$, by $(M)$ and $(H)$ we have for $\tau \in [2,
2_s^\ast]$,
\begin{eqnarray}\label{e5.11}
\kappa\varepsilon^{N-2s}  &\geq&\nonumber J_\varepsilon
(u_\varepsilon) = J_\varepsilon(u_\varepsilon) -
\frac{1}{\tau}J_\varepsilon'(u_\varepsilon)u_\varepsilon
\\
&=&\nonumber
\frac12 \widetilde{M}\left([u_\varepsilon]_{s,A_\varepsilon}^2\right)-
\frac{1}{\tau} M\left([u_\varepsilon]_{s,A_\varepsilon}^2\right)[u_\varepsilon]_{s,A_\varepsilon}^2+\left(\frac{1}{2}-\frac{1}{\tau}\right)\varepsilon^{-2s}
\int_{\mathbb{R}^N}V(x)|u_\varepsilon|^2dx \\
&&\nonumber \mbox{}+
\left(\frac{1}{\tau}-\frac{1}{2_s^\ast}\right)\varepsilon^{-2s}\int_{\mathbb{R}^N}
|u_\varepsilon|^{2_s^\ast}dx  + \varepsilon^{-2s}\int_{\mathbb{R}^N}
\left[\frac{1}{\tau}h(x,
 |u_\varepsilon|^2)u_\varepsilon-\frac{1}{2}H(x,|u_\varepsilon|^2)\right]dx\\
&\geq&\nonumber
\left(\frac{\sigma}{2}-\frac{1}{\tau}\right)\alpha_0[u_\varepsilon]_{s,A_\varepsilon}^2
+ \left(\frac{1}{2}-\frac{1}{\tau}\right)\varepsilon^{-2s}
\int_{\mathbb{R}^N}V(x)|u_\varepsilon|^2dx\\
&& \mbox{}+
\left(\frac{1}{\tau}-\frac{1}{2_s^\ast}\right)\varepsilon^{-2s}\int_{\mathbb{R}^N}
|u_\varepsilon|^{2_s^\ast}dx +
\left(\frac{\mu}{\tau}-\frac{1}{2}\right)\varepsilon^{-2s}
\int_{\mathbb{R}^N}H(x,|u_\varepsilon|^2)dx.
\end{eqnarray}
Taking $\tau = 2/\sigma$, we obtain the estimate $\eqref{e1.8}$ and
taking $\tau = \mu$ we obtain the estimate
$\eqref{e1.9}$. This completes the proof of the first part of Theorem 3.1.

\vspace{2mm}

\noindent{\bf Proof of Theorem 3.1 (2).}
 Denote the set of all symmetric
(in the sense that $-Z = Z$) and closed subsets of $E$
by $\Sigma$.
 For each $Z \in \Sigma$, let gen$(Z)$ be the Krasnoselski genus and
\begin{displaymath}
j(Z) := \min_{\iota\in
\Gamma_{m^\ast}}\mbox{gen}(\iota(Z)\cap\partial
B_{\varrho_\varepsilon}),
\end{displaymath}
where $\Gamma_{m^\ast}$ is the set of all odd homeomorphisms $\iota
\in C(E, E)$ and $\varrho_\varepsilon$ is the number from Lemma
\ref{lemma4.1}. Then $j$ is a version of Benci's pseudoindex
\cite{b1}. Let
\begin{displaymath}
c_{\varepsilon i} := \inf_{j(Z)\geq i}\sup_{u\in Z}J_\varepsilon(u),
\quad 1 \leq i \leq m^\ast.
\end{displaymath}
Since $J_\varepsilon(u) \geq \alpha_\varepsilon$ for all $u \in
\partial B_{\varrho_\varepsilon}^{+}$ and $j(F_{\varepsilon m^\ast}) = m^\ast=
\dim F_{\varepsilon m^\ast}$, we obtain by Lemma 5.5
\begin{displaymath}
\alpha_\varepsilon \leq c_{\varepsilon 1} \leq \cdots\leq
c_{\varepsilon m^\ast} \leq \sup_{u \in F_{\varepsilon m^\ast}}
J_\varepsilon(u) \leq \kappa\varepsilon^{N-2s}.
\end{displaymath}
It follows from Theorem \ref{lemma3.4}  that $J_\varepsilon$ satisfies
the $(PS)_{c_\varepsilon}$ condition at all levels $c <
\sigma_0\varepsilon^{N-2s}$. By the usual critical point theory, all
$c_{\varepsilon i}$ are critical levels and $J_\varepsilon$ has at
least $m^\ast$ pairs of nontrivial critical points satisfying
\begin{displaymath}
\alpha_\varepsilon \leq J_\varepsilon(u_\varepsilon) \leq
\kappa\varepsilon^{N-2s}.
\end{displaymath}
Hence, problem \eqref{e3.1} has at least $m^\ast$ pairs of
solutions. Finally, as in the proof of the first of Theorem 3.1, we see that
these solutions satisfy the estimates  $\eqref{e1.8}$ and
$\eqref{e1.9}$. This completes the proof of the second part of Theorem 3.1. $\hfill\Box$

\section*{Acknowledgements}
We thanks the referees for useful comments and suggestions. S. Liang was supported by
National Natural Science Foundation of China (No. 11301038),
the Natural Science Foundation of Jilin Province (No. 20160101244JC).
D. Repov\v{s} was supported by the Slovenian Research Agency
(No. P1-0292, J1-7025, and J1-8131).
B. Zhang was supported by National Natural Science Foundation of China (No. 11601515, 11701178), 
Natural Science Foundation of Heilongjiang Province of China (No. A201306) and Research Foundation of Heilongjiang Educational Committee (No. 12541667).


\begin{thebibliography}{99}
%\bibitem{a1} {A. Ambrosetti, M. Badiale, S. Cingolani, Semiclassical statesof nonlinear Schr$\ddot{\mbox{o}}$dinger equations, Arch. Ration.Mech. Anal. 140 (1997) 285-300.}

%\bibitem{a2} {A. Ambrosetti, A. Malchiodi, Perturbation Methods and SemilinearElliptic Problems on $\mathbb{R}^N$, Progr. Math., vol. 240,Birkh$\ddot{\mbox{a}}$user, Basel, 2006.}

%\bibitem{a3} {A. Ambrosetti, A. Malchiodi, S. Secchi, Multiplicityresults for some nonlinear Schr$\ddot{\mbox{o}}$dinger equationswith potentials, Arch. Ration. Mech. Anal. 159 (2001) 253-271.}


%\bibitem{a4} {A. Ambrosetti, V. Felli, A. Malchiodi, Ground states of nonlinearSchr$\ddot{\mbox{o}}$dinger equations with potentials vanishing atinfinity, J. Eur. Math. Soc.  7 (2005) 117-144.}

\bibitem{ADM} C.O. Alves, J.M. do \'{O}, O.H. Miyagaki,
Concentration phenomena for fractional elliptic equations involving exponential critical growth,  Adv. Nonlinear Stud.
16 (2016) 843--861.

\bibitem{AH}
V. Ambrosio, H. Hajaiej,
Multiple solutions for a class of nonhomogeneous fractional Schr\"{o}dinger equations in $\mathbb{R}^{N}$,
Journal of Dynamics and Differential Equations, doi: 10.1007/s10884-017-9590-6.

\bibitem{AAHM}
P. Antonelli, A. Athanassoulis, H. Hajaiej, P. Markowich,
On the XFEL Schr\"{o}dinger equation: highly oscillatory magnetic potentials and time averaging,
Arch. Ration. Mech. Anal. 211 (2014) 711--732.


\bibitem{ap} {D. Applebaum, L$\acute{\mbox{e}}$vy processes-from probalility to finance and quantum
groups, Notices Amer. Math. Soc. 51 (2004) 1336--1347.}

\bibitem{au} { G. Autuori, A. Fiscella, P. Pucci, Stationary Kirchhoff problems
involving a fractional operator and a critical nonlinearity,
Nonlinear Anal. 125 (2015) 699--714.}

\bibitem{b2} R. Bartolo, A. Fiscella, Multiple solutions for a class of Schr\"{o}dinger
equations involving the fractional $p$--Laplacian, Minimax Theory Appl. 2 (2017)
9--25.

%\bibitem{b2} {B. Barrios, E. Colorado, R. Servadei, F. Soria, A critical fractional equation with concave-convex power nonlinearities,Ann. Inst. H. Poincar$\acute{e}$ Anal. 32 (2015) 875--900.}

\bibitem{b1} {V. Benci, On critical point theory of indefinite functionals in the presence of symmetries, Trans.
Amer. Math. Soc. 274 (1982) 533--572.}


\bibitem{BMS}
Z. Binlin, G. Molica Bisci, R. Servadei, Superlinear nonlocal fractional problems with infinitely many solutions,
Nonlinearity 28 (2015) 2247--2264.


\bibitem{zhang3} {Z. Binlin, M. Squassina, Z. Xia, Fractional NLS equations with
magnetic field, critical frequency and critical
growth, Manuscripta Math. doi: 10.1007/s00229-017-0937-4.}


\bibitem{Bog} V.I. Bogachev, Measure Theory, Vol. II, Springer-Verlag, Berlin, 2007.


\bibitem{BNV} D. Bonheure, M. Nys, J. Van Schaftingen,
Properties of ground states of nonlinear Schr\"{o}dinger equations under a weak constant magnetic field,
submitted. arXiv:1607.00170.




\bibitem{bre} {H. Br\'{e}zis, L. Nirenberg, Positive solutions of nonlinearelliptic equations involving critical exponents, Commun. Pure Appl. Math. 34 (1983) 437--477.}





\bibitem{ca} {L. Caffarelli, L. Silvestre, An extension problem related to the
fractional Laplacian, Comm. Partial Differential Equations 32 (2007)
1245--1260.}

%\bibitem{d1} {M. Del Pino, P. Felmer, Multi-peak bound states fornonlinear Schr$\ddot{\mbox{o}}$dinger equations, J. Funct. Anal. 149(1997) 245-265.}

%\bibitem{d2} {M. Del Pino, P. Felmer,Semi-classical states for nonlinear Schr$\ddot{\mbox{o}}$dingerequations, Ann. Inst. H. Poincar$\acute{e}$ 15 (1998) 127-149.}

%\bibitem{de3} {M. del Pino, P.L. Felmer, Local mountain passes for semilinearelliptic problems in unbounded domains, Calc. Var. PartialDifferential Equations 4 (1996) 121-137.}


\bibitem{fel} {P. d'Avenia, M. Squassina, Ground states for fractional magnetic operators, ESAIM:
Control Optim. Calc. Var. doi: 10.1051/cocv/2016071.}

\bibitem{di} { E. Di Nezza, G. Palatucci, E. Valdinoci, Hitchhiker's guide to the fractional Sobolev spaces, Bull. Sci. Math. 136 (2012)
521--573.}


\bibitem{DV} J. Di Cosmo, J. Van Schaftingen, Semiclassical stationary states for nonlinear Schr\"{o}dinger equations under a strong external magnetic field,  J. Differential Equations 259 (2015) 596--627.


\bibitem{d3} {Y.H. Ding, F.H. Lin, Solutions of perturbed
Schr$\ddot{\mbox{o}}$dinger equations with critical nonlinearity,
Calc. Var. Partial Differential Equations 30 (2007) 231--249.}

\bibitem{DW}
Y.H. Ding, Z.-Q. Wang, Bound states of nonlinear Schr\"{o}dinger equations with magnetic fields, Annali di Matematica 190 (2011) 427--451.

\bibitem{we1} { W. Dong, J. Xu, Z. Wei, Infinitely many weak solutions for a
fractional Schr$\ddot{\mbox{o}}$dinger equation, Boundary Value
Problems (2014) 2014 159.}

\bibitem{feng} { B.H. Feng, Ground states for the fractional Schr$\ddot{\mbox{o}}$dinger equation,
Electron. J. Differential Equations 2013 (2013) 1--11.}

\bibitem{fig} {G.M. Figueiredo, G.Molica Bisci, R. Servadei, On a fractional
Kirchhoff-type equation via Krasnoselskii's genus, Asymptotic
Anal. 94 (2015) 347--361.}





\bibitem{f2} A. Fiscella, A. Pinamonti, E. Vecchi, Multiplicity results for magnetic
fractional problems, J. Differential Equations 263 (2017) 4617--4633.


\bibitem{f3} A. Fiscella, P. Pucci, $p$--fractional Kirchhoff equations involving
critical nonlinearities, Nonlinear Anal. Real World Appl. 35 (2017) 350--378.

\bibitem{fi} {A. Fiscella, E. Valdinoci, A critical Kirchhoff type problem
involving a nonlocal operator, Nonlinear Anal. 94 (2014) 156--170.}


\bibitem{FTRV} S. Fournais, L. L. Treust, N. Raymond, J. Van Schaftingen,
Semiclassical Sobolev constants for the electro-magnetic Robin Laplacian,
to appear in J. Math. Soc. Japan.


%\bibitem{f1} {A. Floer, A. Weinstein, Nonspreading wave packets for the cubicSchr$\ddot{\mbox{o}}$dinger equation with a bounded potential, J.Funct. Anal. 69 (1986) 397-408.}

\bibitem{liang2} {S. Liang, S. Shi,   Soliton solutions to Kirchhoff type problems involving the critical growth in $\mathbb{R}^N$, Nonlinear Anal.  81 (2013) 31--41.}

\bibitem{liang1} {S. Liang, J. Zhang,   Existence of solutions for Kirchhoff type problems with critical
nonlinearity in $\mathbb{R}^3$, Nonlinear Anal. Real World
Appl.  17 (2014) 126--136.}



\bibitem{lions1} {P.L. Lions, The concentration compactness principle in the calculus
of variations. The locally compact case. Parts I and II, Ann. Inst.
H. Poincare Anal. Non. Lineaire. 1 (1984). pp. 109--145 and 223--283.}


 \bibitem{MPSZ} X. Mingqi, P. Pucci, M. Squassina, B.L. Zhang, Nonlocal Schr$\ddot{\mbox{o}}$dinger--Kirchhoff equations with external magnetic field,
Discrete Contin. Dyn. Syst. 37 (2017) 503--521.


\bibitem{m4} {G. Molica Bisci, V. R$\breve{\mbox{a}}$dulescu, R. Servadei, Variational Methods for
Nonlocal Fractional Problems, Cambridge University Press, Cambridge,
2016.}

\bibitem{m0} {G. Molica Bisci, D. Repov$\check{\mbox{s}}$, Higher nonlocal problems with
bounded potential, J. Math. Anal. Appl. 420 (2014) 167--176.}




\bibitem{m2} {G. Molica Bisci, D. Repov\v{s}, On doubly nonlocal fractional elliptic equations, Rend. Lincei Mat. Appl. 26 (2015) 161--176.}

%\bibitem{m1} {G. Molica Bisci, Sequences of weak solutions for fractional equations, Math. Res. Lett. 21 (2014) 241--253.}




%\bibitem{m3} {G. Molica Bisci, R. Servadei, Lower semicontinuity of functionals offractional type and applications to nonlocal equations with criticalSobolev exponent, Adv. Differential Equations 20 (2015) 635--660.}





\bibitem{PP}
G. Palatucci, A. Pisante, Improved Sobolev embeddings, profile decomposition, and concentration-compactness for fractional Sobolev spaces. Calc. Var. Partial Differential Equations 50 (2014) 799--829.



\bibitem{PSV2} A. Pinamonti, M. Squassina, E. Vecchi, The Maz'ya-Shaposhnikova
limit in the magnetic setting, J. Math. Anal. Appl. 449 (2017) 1152--1159.

\bibitem{PSV1} A. Pinamonti, M. Squassina, E. Vecchi, Magnetic BV functions
and the Bourgain-Brezis-Mironescu formula, Adv. Calc. Var.
doi: 10.1515/acv-2017-0019.




\bibitem{pu} {P. Pucci, S. Saldi, Critical stationary Kirchhoff equations in
$\mathbb{R}^N$ involving nonlocal operators, Rev. Mat. Iberoam. 32
(2016) 1--22.}




\bibitem{PXZ} P. Pucci, M.Q. Xiang, B.L. Zhang, Multiple solutions for nonhomogeneous Schr\"{o}dinger-Kirchhoff type equations involving the fractional $p$--Laplacian in ${\mathbb {R}}^N$, Calc. Var. Partial Differetial Equations 54 (2015) 2785--2806.

\bibitem{PXZ1} P. Pucci, M.Q. Xiang, B.L. Zhang, Existence and multiplicity of entire solutions for fractional $p$--Kirchhoff equations, Adv. Nonlinear Anal. 5 (2016) 27--55.


\bibitem{r1} {P.H. Rabinowitz,  Minimax methods in critical-point
theory with applications to differential equations, CBME Regional
Conference Series in Mathematics, Volume 65 (American Mathematical
Society, Providence, RI, 1986).}

%\bibitem{ri1} {B. Ricceri, A multiplicity result for nonlocal problems involving nonlinearities with bounded primitive, Studia Univ. Babes-Bolyai Math. 55 (2010) 107-114.}


\bibitem{RS} M. Reed, B. Simon, Methods of Modern Mathematical Physics, IV. Analysis of Operators, Academic
Press, London, 1978.





\bibitem{sec} {S. Secchi, Ground state solutions for nonlinear fractional
Schr$\ddot{\mbox{o}}$dinger equations in $\mathbb{R}^N$, J. Math.
Phys. 54 (2013) 031501.}

%\bibitem{ser1} {R. Servadei, E. Valdinoci, The Br$\acute{\mbox{e}}$zis-Nirenberg result for thefractional Laplacian, Trans. Amer. Math. Soc. 367 (2015) 67--102.}

%\bibitem{ser2} {R. Servadei, E. Valdinoci, Fractional Laplacian equations withcritical Sobolev exponent, Rev. Mat. Complut. 28 (2015) 655--676.}

%\bibitem{MSBV}M. Squassina, B. Volzone,Bourgain-Brezis-Mironescu formula for magnetic operators,{\em C. R. Math. Acad. Sci. Paris}  354 (2016) 825--831.

\bibitem{sh1} {X. Shang, J. Zhang, Ground states for fractional Schr$\ddot{\mbox{o}}$dinger
equations with critical growth, Nonlinearity 27 (2014) 187--207.}

\bibitem{sh2} {X. Shang, J. Zhang, Y. Yang, On fractional Schr$\ddot{\mbox{o}}$dinger equation in $\mathbb{R}^N$ with critical growth, J. Math. Phys. 54 (2013)
121502.}

\bibitem{sh3} {Z. Shen, F. Gao, On the existence of solutions for the critical
fractional Laplacian equation in $\mathbb{R}^N$, Abstr. Appl. Anal.
2014 (2014) 1--10.}


\bibitem{sq1} {M. Squassina, B. Volzone, Bourgain-Br\'{e}zis-Mironescu formula for
magnetic operators,  C. R. Math. 354 (2016) 825--831.}

\bibitem{te1} {K. Teng, X. He, Ground state solutions for fractional Schr$\ddot{\mbox{o}}$dinger
equations with critical Sobolev exponent, Commun. Pure Appl. Anal.
15 (2016) 991--1008.}



\bibitem{WX} F. Wang, M.Q. Xiang, Multiplicity of solutions to a nonlocal Choquard equation involving fractional magnetic operators and critical exponent, Electron. J. Differ. Eq. 2016 (2016) 1--11.
\bibitem{w1} {M. Willem, Minimax Theorems, Birkh$\ddot{a}$ser,
Boston, 1996.}



\bibitem{XZR}{ M.Q. Xiang, B.L. Zhang,  V. R\u{a}ulescu, Multiplicity of solutions for a class of quasilinear Kirchhoff
system involving the fractional $p$--Laplacian, Nonlinearity 29 (2016)
3186--3205.}

\bibitem{XZZ} M.Q. Xiang, B.L. Zhang, X. Zhang, A nonhomogeneous fractional $p$--Kirchhoff type problem  involving critical exponent in $\mathbb{R}^N$, Adv. Nonlinear Stud. 17 (2017) 611--640.


\bibitem{YW} J.F. Yang, F.J. Wu, Doubly critical problems involving fractional
laplacians in $\mathbb{R}^N$,  Adv. Nonlinear Stud. 17 (2017) 677--690.



\bibitem{ZDM} J.J. Zhang, J.M. do \'{O}, M. Squassina,
Fractional Schr$\ddot{\mbox{o}}$dinger--Poisson systems with a general subcritical or critical nonlinearity,
Adv. Nonlinear Stud. 16 (2016) 15--30.



\bibitem{zhang2} {X. Zhang, B.L. Zhang, D. Repov\v{s}, Existence and symmetry of solutions for critical fractional
Schr$\ddot{\mbox{o}}$dinger equations with bounded potentials,
Nonlinear Anal. 142 (2016) 48--68.}

\bibitem{zhang1} { X. Zhang, B.L. Zhang, M.Q. Xiang, Ground states for fractional
Schr$\ddot{\mbox{o}}$dinger equations involving a critical
nonlinearity, Adv. Nonlinear Anal.  5 (2016) 293--314.}







%\bibitem{z1} { H. Zou, Existence and non-existence for Schr$\ddot{\mbox{o}}$dinger equationsinvolving critical Sobolev exponents, J. Korean Math. Soc. 47 (2010)547-572.}





\end{thebibliography}
\end{document}